\documentclass{article}
\usepackage[utf8]{inputenc}
\usepackage[utf8]{inputenc}
\usepackage{tikz}
\usepackage{authblk}
\usepackage[caption = false]{subfig}
\usepackage[a4paper, total={7in, 9in}]{geometry}

\usepackage{comment}
\usepackage{amsmath}
\numberwithin{equation}{section}
\usepackage{amsthm}
\usepackage{amssymb}
\newtheorem{theorem}{Theorem}[subsection]

\newtheorem{proposition}{Proposition}[subsection]
\newtheorem{remark}{Remark}[subsection]
\newtheorem{example}{Example}[subsection]
\newtheorem{definition}{Definition}[subsection]

\usepackage[
backend=biber,
style=numeric,
sorting=ynt
]{biblatex}
\title{A Generalization of Bivariate Lack-of-Memory Properties }
\author[1]{Massimo Ricci}
\affil[1]{Department of Statistical Sciences, Via Belle Arti 41, Bologna, Italy}
\begin{document}
\maketitle
\abstract{
\noindent In this paper, we propose an extension of the standard bivariate lack-of-memory properties, replacing in the associated functional equations the standard product by the pseudo one, defined as $a\otimes_hb=h\left (h^{-1}(a)\cdot h^{-1}(b)\right ), \ \forall \ a,b \in [0,1],$ where $h$ is an increasing bijection of $[0,1]$, called generator.
We say that a survival function $\bar{F}$ satisfies pseudo lack-of-memory property in strong version if 
$\bar F(s_1+t_1,s_2+t_2)=\bar F_{X,Y}(s_1,s_2)\otimes_h\bar F_{X,Y}(t_1,t_2), \ \forall \ t_1,t_2,s_1,s_2 \geq 0$
and in weak version if
$\bar F(s_1+t,s_2+t)=\bar F_{X,Y}(s_1,s_2)\otimes_h\bar F_{X,Y}(t,t), \ \forall \ t,s_1,s_2 \geq 0.$
After finding sufficient and necessary conditions under which the solutions of the above functional equations are bivariate survival functions, we focus our attention on the distribution which satisfies the pseudo weak version. In particular, we provide a characterization, in line with the existing literature, of that distribution 
and  we investigate its dependence structure, analysing the impact that the generator $h$ has on its Kendall distribution function and on its tail dependence coefficients.
\vskip 0.3cm
\section{Introduction}
In Marshall and Olkin (1967), the following stochastic representation is specified for the random vector 
$(X,Y)$:
\begin{equation}
(X, Y) = (\min(T_1, T_3),\min(T_2, T_3)),
\label{intr1}
\end{equation}
where $T_1$, $T_2$ are independent non-negative continuous random variables representing the individual shocks while $T_3$, independent of $T_1$ and $T_2$, is the common shock affecting both variables $X$ and $Y$.
For example, in reliability theory, (\ref{intr1}) may represent the vector of lifetimes of two
components subjected to a common environmental shock, while, in the
theory of credit risk, it can be considered as the vector of times to default of two counter-parties which face both idiosyncratic and systemic risks. \\
In the case in which $T_i$ in the stochastic representation (\ref{intr1}) are independent and exponentially distributed random variables with parameters $\lambda_i$, $i=1,2,3$, the distribution of the vector $X,Y$ is the famous bivariate exponential Marshall-Olkin distribution with survival function 
\begin{equation}
\bar{F}_{X,Y}(x, y) = P(X > x, Y > y) = e^{-\lambda_1  x - \lambda_2  y -\lambda_3 \max\{x,y\}}
\label{intr2}
\end{equation}
with $x, y \geq 0$ and $\lambda_i \geq 0, i = 1, 2, 3$. \\
However, the model above shows many limitations, including the fact that marginal hazard rates are constant: for this reason, some authors assumed that the random variables $T_i$ above follow a second-type Pareto distribution or the Weibull distribution in order to model bivariate survival data ( see among the others Asimit et al., 2010, and Lu,1989). \\
On the other hand, many authors focus on the generalization of the lack-of-memory property introduced by Marshall and Olkin. For example, it is well known that the survival function (\ref{intr2}) is the unique bivariate survival function with exponential marginals satisfying the bivariate weak lack-of-memory property,
\begin{equation}
    \bar{F}_{X,Y}(s_1+t,s_2+t) = \bar{F}_{X,Y}(s_1,s_2) \cdot \bar{F}_{X,Y}(t,t)
    \label{intr4}
\end{equation}
for all $s_1, s_2 \geq 0$ and $t \geq 0$. 
The strong version of (\ref{intr4}) is
\begin{equation}
    \bar{F}_{X,Y}(s_1+t_1,s_2+t_2) = \bar{F}_{X,Y}(s_1,s_2) \cdot \bar{F}_{X,Y}(t_1,t_2), \ s_1,s_2,t_1,t_2 \geq 0
    \label{intr5}
\end{equation}
and, in Marshall and Olkin (1967), it has been showed that the unique solution of the latter is given by 
\begin{equation}
    \Bar{F}_{X,Y}(x,y) = e^{-\lambda_1 x - \lambda_2 y}, \ \lambda_1, \lambda_2 > 0.
    \label{strong standard solution}
\end{equation}
In Muliere and Scarsini (1987), the authors generalize the functional equation (\ref{intr4}) by substituting the standard sum $+$ by an associative and reducible binary operator $\ast$ such that
$$x \ast y = g(g^{-1}(x) + g^{-1}(y)), x, y \in \mathbf{R}_+ $$
where $g$ is a continuous and monotone function.
The distribution, that they call Marshall-Olkin type distribution, that satisfies the functional equation 
\begin{equation*}
    \bar{F}_{X,Y}(s_1 \ast t,s_2 \ast t) = \bar{F}_{X,Y}(s_1,s_2) \cdot \bar{F}_{X,Y}(t,t),
\end{equation*}
has proved to be rather flexible in practice, including several useful bivariate distributions. 
A generalization of bivariate Marshall-Olkin distribution, which includes the distribution obtained by Muliere and Scarsini as a particular case, has been performed also in Li and Pellerey (2011), in which $T_i$ in (\ref{intr1}) are assumed only  to be independent but with general marginal survival functions.\\
In the spirit of Muliere and Scarsini, in this paper, we replace the standard product in (\ref{intr4}) and (\ref{intr5}) by a binary operator which can be seen as particular T-norm, obtaining the following functional equations: 
\begin{equation}
    \bar{F}_{X,Y}(s_1 + t,s_2 + t) = \bar{F}_{X,Y}(s_1,s_2) \otimes_h \bar{F}_{X,Y}(t,t), s_1,s_2,t \geq 0
    \label{intr_pseudo_weak}
\end{equation}
and
\begin{equation}
    \bar{F}_{X,Y}(s_1 + t,s_2 + t) = \bar{F}_{X,Y}(s_1,s_2) \otimes_h \bar{F}_{X,Y}(t_1,t_2), \ s_1,s_2,t_1,t_2 \geq 0,
    \label{intr_pseudo_strong}
\end{equation}
where $$a \otimes_h b = h(h^{-1}(a) \ h^{-1}(b)), \ a,b \in [0,1].$$
We say that a distribution satisfying (\ref{intr_pseudo_strong}) possesses pseudo strong lack-of-memory property, while a distribution for which (\ref{intr_pseudo_weak}) holds true is said to have pseudo weak lack-of-memory property.\\
\vskip 0.2 cm
After finding sufficient and necessary conditions under which the solutions of (\ref{intr_pseudo_strong}) and of (\ref{intr_pseudo_weak}) are bivariate survival functions, we focus our attention on the solution of (\ref{intr_pseudo_weak}) which proves to be more flexible and we show that the latter may have a singularity along the line $x=y$.\\
In Ghurye and Marshall (1984) and in Block and Basu (1974), characterizations of distributions satisfying standard weak lack-of-memory property (\ref{intr4}) have been provided. In this paper, we analyse possible extensions to distributions satisfying (\ref{intr_pseudo_weak}): in particular, we show that the characterization given in Black and Basu (1974) still holds true in the pseudo weak lack-of-memory property framework.
\vskip 0.2 cm
Finally, we analyse the impact of the generator $h$ on the dependence structure of the distributions satisfying bivariate pseudo weak lack-of-memory property (\ref{intr4}), determining the Kendall distribution function in full generality and
the lower and the upper tail dependence coefficients in some specific cases.\\
The paper is organized as follows.
In section 2, we recall the main results obtained in Marshall and Olkin (1967) and the definition and the representation of a continuous T-norm.
In section 3, we discuss the conditions under which the solutions of (\ref{intr_pseudo_weak}) and (\ref{intr_pseudo_strong}) are bivariate survival functions; for the distribution satisfying (\ref{intr_pseudo_weak}), we analyse how the singularity mass is spread along the line $x=y$.
In section 4, we analyse possible extensions to pseudo-weak distributions of the characterizations provided in Ghurye and Marshall (1984) and in Block and Basu (1974) in the standard weak lack-of-memory framework.
In section 5, we recover an expression for the Kendall distribution function in full generality
and we determine the lower and the upper tail dependence coefficients for specific choices of marginal survival functions and of the generator $h$.
In section 6, we give conclusions.

\section{Preliminaries}
In this section, we present the main concepts about lack-of-memory properties and T-norms.
In the whole paper, we'll deal with positive-valued random variables with absolutely continuous marginal distributions with support $(0,\infty)$, unless differently specified.
\subsection{Standard Lack of Memory Properties}
A survival distribution $\Bar{G}$ satisfies the standard univariate lack-of-memory property if 
\begin{equation}
    \Bar{G}(s+t) = \Bar{G}(s) \Bar{G}(t), \ s,t \geq 0.
    \label{univariate equation}
\end{equation}
The unique survival function satisfying the above functional equation is well known and is given by $\Bar{G}(x) = e^{-\lambda x}, \lambda > 0$. 
\vskip 0.2 cm
Definitions of bivariate lack-of-memory properties have been given in the seminal paper of Marshall and Olkin (1967).\\
A bivariate survival function $\Bar{G}$ satisfies the strong bivariate lack-of-memory property if 
\begin{equation}
    \bar{G}(s_1+t_1,s_2+t_2) = \bar{G}(s_1,s_2) \Bar{G}(t_1,t_2), \ s_1, s_2, t_1, t_2 \geq 0,
    \label{strong standard equation}
\end{equation}
and the survival function satisfying the above functional equation is given by $\Bar{G}(x,y) = e^{-\lambda_1 x - \lambda_2 y}, \ \lambda_1, \lambda_2 > 0.$\\
Similarly, a bivariate survival function $\Bar{G}$ satisfies the weak bivariate lack-of-memory property if 
\begin{equation}
    \bar{G}(s_1+t,s_2+t) = \bar{G}(s_1,s_2) \ \Bar{G}(t,t), \ s_1, s_2, t \geq 0.
    \label{weak standard equation}
\end{equation}
The solution of the above functional equation is given by
\begin{equation}
    \Bar{G}(x,y) = 
    \begin{cases}
    e^{-\lambda y} \Bar{G}_1(x-y), \ x \geq y \\
     e^{-\lambda x} \Bar{G}_2(y-x), \ x < y
    \end{cases},
    \label{weak standard solution}
\end{equation}
where $\Bar{G}_1$, $\Bar{G}_2$ are univariate  survival functions of positive and absolutely continuous random variables and $\lambda$ is a positive constant.
However, the function (\ref{weak standard solution}) is a survival function with absolutely continuous marginal distributions if and only if 
\begin{equation}
    \begin{cases}
        \lambda \leq g_1(0)+g_2(0) \\
        \frac{\partial \log (g_i(z))}{\partial z}  \geq -\lambda, \ \forall z \geq 0, \ i = 1,2 
    \end{cases},
    \label{system}
\end{equation}
where $g_i(z) = \frac{-\partial \Bar{G}_i(x)}{\partial x}$, see Theorem 5.1 in Marshall and Olkin (1967).
The two conditions of the system above guarantee that the probability mass on the line $x=y$ is $\frac{g_1(0)+g_2(0)}{\lambda}-1 \in [0,1]$ and that the function (\ref{weak standard solution}) is absolutely continuous in the region $\{(x,y): x \geq 0, y \geq 0, x \neq y\}$.
\subsection{T-norm}
Aim of this paper is to generalize (\ref{strong standard equation}) and (\ref{weak standard equation}) substituting the standard product by a T-norm (see, among the others, Klement et al., 2004).
\begin{definition}
A T-norm can be defined as a mapping $T:  [0,1] \times [0,1] \rightarrow[0,1]$ such that: $$T(a,b)=T(b,a);$$
$$T(a,b) \geq T(c,d), \ a \geq c, \ b \geq d;$$
$$T(a,T(b,c))=T(T(a,b),c);$$
$$T(1,a)=a,$$
where $a,b,c,d \in [0,1]$.
\label{T Norm}
Moreover, a T-norm is strict if it is continuous and strictly monotone.
Finally, a T-norm $T$ is Archimedean if and only if it is continuous and $T(x,x) < x, \ x \in (0,1)$.
\end{definition}
The following representation Theorem for a T-norm can be found in Klement et al. (2004).
\begin{theorem}
A function $T$ : $[0, 1] \times [0, 1] \rightarrow{} [0, 1]$ is a strict Archimedean T-norm if and only if there exists a strictly decreasing and continuous function
$t: [0, 1] \rightarrow [0, \infty]$ with $t(1) = 0$ and $t(0) = \infty$, uniquely determined up to a positive multiplicative constant, such that
\begin{equation}
T(a,b) = t^{-1}(t(a) + t(b)), \ \forall a,b \in [0,1].
\label{tnorm_eq}
\end{equation}
\label{theorem t-norm}
\end{theorem}
Notice that any T-norm of type (\ref{tnorm_eq}) can be written in multiplicative form as $T(a,b) = h(h^{-1}(a) h^{-1}(b))$, where $h(x) = t^{-1}(-\log(x)).$ Since a T-norm can be seen as a generalization of the standard product, we will call it "pseudo product" and use the notation $\otimes_h$ to denote
    \begin{equation*}
        a \otimes_h b = h(h^{-1}(a)\ h^{-1}(b)),
    \end{equation*}
    \label{pseudo prod def}
    where $h:[0,1] \rightarrow [0,1]$ is a strictly increasing function with $h(0) = 0$ and $h(1) = 1$, called generator.
Notice that, in the case in which $h = id$, we get the standard product.
\begin{remark}
    The functions $h:[0,1] \rightarrow [0,1]$ and $\hat{h}: [0,1] \rightarrow [0,1]$ generate the same pseudo product, id est 
    $$a \otimes_h b = a \otimes_{\hat{h}} b, \ \forall a,b \in [0,1],$$
    if and only if there exists $\alpha > 0$ such that ${h}(x) = \hat{h}(x^\alpha), \forall x \in [0,1]$.
    This can be easily deduced from the fact that the generator of a strict Archimedean T-norm is defined up to a multiplicative constant.
    \label{rm1}
\end{remark}
For the sake of simplicity, in the rest of this paper, we drop the dependence on $h$ and we'll denote by $a \otimes b$ the pseudo-product $a \otimes_h b$.\\
In line with the generalization of the product, we consider the induced generalization of the Cauchy exponential functional equation 
$$f(x+y) = f(x) \otimes f(y).$$ It can be easily verified that the unique solution of the latter is
$$f(x) = h(a^x), \ a > 0,$$ see Mulinacci and Ricci (2024);
in the particular case in which $a = e^{-1}$, we have
\begin{equation}
\exp_h(x) = h(e^{-x}), \ x \geq 0,
\label{pseudo_exp}
\end{equation}
that will be called "pseudo-exponential function" in the sequel.
\section{Pseudo Lack-of-Memory Properties}
In this section, we generalize univariate and bivariate lack-of-memory properties by substituting in the associated functional equations the standard product by the pseudo one.
\subsection{Univariate Pseudo Lack-of-Memory Property}
\begin{definition}
We say that a survival distribution $\Bar{F}$ satisfies the univariate pseudo lack-of-memory property if
\begin{equation}
    \Bar{F}(s+t) = \Bar{F}(s) \otimes \Bar{F}(t), \ s,t \geq 0.
    \label{pseudo univariate}
\end{equation}
 \\
\end{definition}
Clearly, the pseudo-exponential function in (\ref{pseudo_exp}) is the unique solution of the latter, see Remark \ref{rm1}.
A random variable with survival distribution function (\ref{pseudo_exp}) will be called "pseudo exponential" random variable with parameter $\lambda$.
\begin{remark}
    Notice that any univariate survival distribution satisfies the pseudo lack-of-memory property with respect to a suitable pseudo product. In fact, let $\Bar{H}$ be a univariate survival function: then it satisfies the pseudo univariate lack-of-memory property with respect to $h(x) = \Bar{H}(-\log x)$.
    \label{rm2}
\end{remark}
\subsection{Bivariate Pseudo Strong Lack-of-Memory Property}
\begin{definition}
A bivariate distribution $\Bar{F}$ satisfies the bivariate pseudo strong lack-of-memory property if 
\begin{equation}
    \Bar{F}(s_1+t_1, s_2 + t_2) = \Bar{F}(s_1, s_2) \otimes \Bar{F}(t_1,t_2), \forall s_1,s_2,t_1,t_2 \geq 0.
\end{equation}
    \label{PBLMP_def}
\end{definition}
The solution of the above functional equation is given by 
\begin{equation}
         \Bar{F}(s,t) =  \exp_h(\lambda_1 s + \lambda_2 t), \ \lambda_1, \lambda_2 > 0,
         \label{PBLMP_eq}
\end{equation}
which is a bivariate survival function if and only if $h^{-1}$ is log-concave, that is $\log(h^{-1}(x))$ is concave.
In this case, the associated survival copula is of Archimedean type given by
\begin{equation}
    \bar{C}(u,v) = h(h^{-1}(u) h^{-1}(v)) = u \otimes v.
    \label{PS_Copula}
\end{equation} 
\begin{remark}
Survival functions of type (\ref{PBLMP_eq}) have been identified in Proposition 3.1 in Genest and Kolev (2021) as the class of distributions satisfying Bivariate Law of Uniform Seniority, the latter representing an extension to the bivariate case of the Law of Uniform Seniority introduced in De Morgan (1839).
Distributions of kind (\ref{PBLMP_eq}) represent a generalization of bivariate Schur-Constant distributions recovered when $\lambda_1 = \lambda_2 = 1$ (see, among the others, Chi et al. ,2009).
The authors show that the random variables $\lambda_1 X$ and $\lambda_2 Y$ are identically distributed with convex survival function $\exp_h$ and prove that there exist a random variable $R$, with $P(R \geq r) = \exp_h(r)+r h^{'}(e^{-r}) e^{-r}$, and an uniform random variable $U$, independent of $R$, such that $\lambda_1 X \stackrel{d} = RU$ and $\lambda_2 Y \stackrel{d} = R(1-U).$
\end{remark}
\subsection{Bivariate Pseudo Weak Lack-of-Memory Property}
The generalization of the bivariate weak lack-of-memory property is obtained setting into Definition \ref{PBLMP_def} $t_1 = t_2 = t$.
\begin{definition}
A bivariate distribution $\Bar{F}$ satisfies the bivariate pseudo weak lack-of-memory property if 
\begin{equation}
    \Bar{F}(s_1+t, s_2 + t) = \Bar{F}(s_1, s_2) \otimes \Bar{F}(t,t), \forall s_1,s_2,t \geq 0.
    \label{pseudo weak equation}
\end{equation}
    \label{PWBLMP_def}
\end{definition}
It can be easily verified that the solution of this functional equation is:
\begin{equation}
     \Bar{F}(x,y)  = h(\Bar{G}(x,y)) 
     = \begin{cases}
        h\left(e^{-\lambda y} \Bar{G}_1(x-y)\right) \ x \geq y \\
        h\left(e^{-\lambda x} \Bar{G}_2(y-x)\right) \ x < y
    \end{cases}  = \begin{cases}
        \exp_h(\lambda y)  \otimes \Bar{F}_1(x-y) \ x \geq y \\
        \exp_h(\lambda x)  \otimes \Bar{F}_2(y-x) \ x < y 
    \end{cases},
    \label{PWBLMP_eq}
\end{equation}
where $\bar{G}$ satisfies the standard weak-lack-of-memory property, with $\Bar{G}_1$, $\Bar{G}_2$, $\Bar{F}_1$ and $\Bar{F}_2$ marginal univariate survival functions of non-negative random variables such that $\Bar{F}_i(x) = h(\Bar{G}_i(x)), \ i = 1,2.$

Sufficient and necessary conditions under which (\ref{PWBLMP_eq}) is a bivariate survival function are given in the following Proposition.
\begin{proposition}
     Let $\Bar{F}_1$ and $\Bar{F}_2$ be twice differentiable univariate marginal survival functions of the random variables $X$ and $Y$ and let $h$ be a twice differentiable generator with $h^{'}(x) > 0, \forall x \in [0,1]$. 
     Then (\ref{PWBLMP_eq}) is a survival function if and only if 
     \begin{equation}
\begin{cases}
\frac{\partial^2 \Bar{F}_{X,Y}(x,y)}{\partial x \partial y} \geq 0 \ \forall x \geq 0, \ y \geq 0 \ x \neq y \\
\lambda \leq \frac{f_1(0)+f_2(0)}{h^{'}(1)}
\label{system2}
\end{cases},
\end{equation}
where $f_i(x) = -\bar{F}^{'}_i(x), i = 1,2$.
Moreover, if $ \lambda < \frac{f_1(0)+f_2(0)}{h^{'}(1)}$, the distribution has a singularity along the line $x = y$ with probability mass $\frac{f_1(0)+f_2(0)}{\lambda h^{'}(1)}-1$.
\label{prop1}
\end{proposition}
\begin{proof}
Under the considered assumptions, the second mixed derivative $\frac{\partial^2 \Bar{F}_{X,Y}}{\partial x \partial y}$ is well-defined.
When $x > y$, $\Bar{F}_{X,Y}(x,y) = h(e^{-\lambda y} h^{-1}(\Bar{F}_1(x-y)))$, so 
\begin{equation}
    \begin{split}
    \int_{0}^{\infty} \int_{0}^{x} \frac{\partial^2 \bar{F}_{X,Y}}{\partial x \partial y} \ dy \ dx & = 
     \int_{0}^{\infty} \frac{\partial \Bar{F}_{X,Y}(x,y)}{\partial x} \big{|}_{y = 0}^{x} \ dx = 
     \int_{0}^{\infty} \left(-h'(\exp(-\lambda x)) \frac{\exp(-\lambda x)}{h'(1)} f_1(0)
    +f_1(x)
    \right) \ dx = \\
    & = \frac{h(\exp(-\lambda x)) f_1(0)}{h^{'}(1) \lambda} - \Bar{F_1}(x) \big{|}_{x = 0}^{\infty} =
    1-\frac{f_1(0)}{\lambda h^{'}(1)} \geq 0;
    \end{split}
    \label{condition1}
\end{equation}

by the same way of reasoning, if $x < y$, we have that
\begin{equation}
    \int_{0}^{\infty} \int_{0}^{y} \frac{\partial^2 \bar{F}_{X,Y}}{\partial x \partial y} \ dx \ dy =
    1-\frac{f_2(0)}{\lambda h^{'}(1)} \geq 0.
    \label{condition2}
\end{equation}
Hence, if the vector $(X,Y) \sim F$, it follows that
\begin{equation}
    P(X = Y)  = 1-\left(1-\frac{f_1(0)}{\lambda h^{'}(1)}\right)-\left(1-\frac{f_1(0)}{\lambda h^{'}(1)}\right) 
    = \frac{f_1(0)+f_2(0)}{\lambda h^{'}(1)}-1,
    \label{condition3}
\end{equation}
which is non-negative if and only if the last condition of the system (\ref{system2}) holds true: the fact that it is bounded above from $1$ comes from (\ref{condition1}) and (\ref{condition2}). 
\end{proof}
\vskip 0.2 cm
Notice that (\ref{condition3}) can be rewritten as $$P(X= Y) = \frac{g_1(0)+g_2(0)}{\lambda}-1$$ where $g_i = -\Bar{G}^{'}$, meaning that the singularity mass does not depend on the generator $h$ but only on the marginal densities of $\Bar{G}_{X,Y}$ satisfying functional equation (\ref{weak standard equation}) from which $\Bar{F}_{X,Y}$ is obtained as a distortion;
however, the choice of $h$ influences how that mass is distributed along the line $x = y$, as we show in next Proposition.
\begin{proposition}
    Let $\Bar{F}_{X,Y}$ be a survival function satisfying Definition \ref{PWBLMP_def}. Then, under the same assumptions of Proposition \ref{prop1}, we have 
    \begin{equation}
    P(X = Y, X \geq t) = 
    \exp_h(\lambda t) \left(\frac{f_1(0)+f_2(0)}{\lambda h^{'}(1)}-1\right).
    \label{eq prop2}
    \end{equation}
\end{proposition}
\begin{proof}
    Since
    $$P( X = Y, X \geq t) = \Bar{F}_{X,Y}(t,t)-P(X > Y \geq t)-P(Y > X \geq t) $$
    and $\Bar{F}_{X,Y}(t,t) = \exp_h(\lambda t)$,
    we have
    \begin{equation*}
        \begin{split}
            P(X > Y \geq t) & = \int_{t}^{+ \infty} \int_{t}^{x} 
            \frac{\partial^2 \Bar{F}_{X,Y}(x,y)}{\partial x \partial y} \ dy \ dx = \int_{t}^{\infty} \frac{\partial \Bar{F}_{X,Y}(x,y)}{\partial x} \big{|}_{y = t}^{x} \ dx = \\ & = 
             \int_{t}^{\infty}
            -h^{'}(e^{-\lambda x}) e^{-\lambda x} \frac{f_1(0)}{h^{'}(1)} \ dx+  
             h^{'}(e^{-\lambda t} h^{-1}(\Bar{F}_1(x-t))) \ e^{-\lambda t} \frac{f_1(x-t)}{h^{'}(h^{-1}(\Bar{F}_1(x-t)))} \ dx  = \\
             & = h(e^{-\lambda t}) \left(1-\frac{f_1(0)}{\lambda h^{'}(1)}\right).
        \end{split}
    \end{equation*}
    By the same way of reasoning,
    $$P(Y > X \geq t) = h(e^{-\lambda t}) \left(1-\frac{f_2(0)}{\lambda h^{'}(1)}\right),$$
    from which the conclusion follows trivially.
\end{proof}
\vskip 0.2 cm
In general, it is difficult to check when (\ref{PWBLMP_eq}) is a bivariate survival function;
however, thanks to Sklar Theorem, (\ref{PWBLMP_eq}) is a joint survival distribution if and only if $\Bar{C}(u,v) = \Bar{F}_{X,Y}(\Bar{F}_1^{-1}(u), \Bar{F}_2^{-1}(v))$ is a copula.
Notice that 
\begin{equation}
\begin{split}
             & \Bar{C}^{F}(u,v) = \\
             & = 
              h(e^{-\lambda \Bar{F_2}^{-1}(v)} \bar{F}_1(\Bar{F_1}^{-1}
            (u)-\bar{F}_2^{-1}(v)) 1_{\Bar{F_1}^{-1}
            (u)-\bar{F}_2^{-1}(v) \geq 0} + e^{-\lambda \Bar{F}_1^{-1}(u)} \bar{F}_2(\Bar{F_2}^{-1}
            (v)-\bar{F}_1^{-1}(u)) 1_{\Bar{F_2}^{-1}
            (v)-\bar{F}_1^{-1}(u) > 0}) = \\
            & = h(\bar{C}^{G}(h^{-1}(u), h^{-1}(v))),
\end{split}
        \label{distortion copula}
\end{equation}
where 
\begin{equation}
\bar{C}^G(u,v)
 = e^{-\lambda \bar{G}_2^{-1}(v)} \bar{G}_1(\bar{G}_1^{-1}(u)-\bar{G}_2^{-1}(v)) 1_{\bar{G}_1^{-1}(u) \geq \bar{G}_2^{-1}(v)}
 + 
e^{-\lambda \bar{G}_1^{-1}(u)} \bar{G}_2(\bar{G}_2^{-1}(v)-\bar{G}_1^{-1}(u)) 1_{\bar{G}_2^{-1}(v) > \bar{G}_1^{-1}(u)}.
\label{C^G}
\end{equation}
$\bar{C}^{\bar{F}}$ is a copula if $\bar{C}^{\bar{G}}$ is a copula and if $h$ is convex, see Klement et al. (2004).
In Ghiselli Ricci (2024), the author provides sufficient conditions under which a distortion of a copula is still a copula: the results therein can be applied to our framework in the following ways.
\begin{enumerate}
    \item If $h$ and $h^{-1}$ are absolutely continuous, and if $\frac{\partial \bar{C}^{\bar{F}}(u,v)}{\partial u}$ is increasing in $v$ almost everywhere, then $\bar{C}^{\bar{F}}$ is a copula, see Proposition 6.4 in Ghiselli Ricci (2024).
    \item Let $\mu_{\bar{C}^{\bar{G}},u}(t) = \Bar{C}^{\Bar{G}}(t,u)$ \color{black}, $\Psi(u,v) = \frac{\frac{\partial \Bar{C}^{\bar{G}}(u,v)}{\partial u}}{\bar{C}^{\bar{G}}(u,v)}$ and $\sigma_h(x) = \frac{h^{-1}(x)}{(h^{-1})'(x)}$.
 If $\mu^{-1}_{\bar{C}^{\bar{G}},u}$ is absolutely continuous for any $u > 0$, if $\Psi(u,v)$ is strictly increasing in second place and if $\sigma_h$ is increasing, then $\Bar{C}^{\bar{F}}$ is a copula, see Theorem 6.6 in Ghiselli Ricci (2024).
\end{enumerate}
\subsubsection{Examples and Specific Cases} 
The fact that $\Bar{F}_{X,Y}$ is a bivariate survival function does not imply that $\bar{G}_{X,Y} = h^{-1}(\bar{F}_{X,Y})$ is a bivariate survival function, as we show in the next Example.
\begin{example}
Let $h(x) = \left(\frac{e^{\theta x}-1}{e^{\theta}-1}\right)^{\beta}$ and let $\bar{F}_1(x) = \bar{F}_2(x) = h((1+x)^{-\alpha})$, with $\alpha > 0$, $\beta > 1$, $\theta > 0$ and
$\max\left(\alpha + \frac{1}{\beta},\frac{\alpha+1+\alpha \beta}{\beta+1} \right)\leq \lambda \leq \min(\alpha+1,2 \alpha)$.
Then the function $\bar{F}_{X,Y}$ satisfying definition \ref{PWBLMP_def} with marginals $\bar{F}_1$, $\bar{F}_2$ is given by 
\begin{equation*}
    \bar{F}_{X,Y}(x,y)
     = \left(\frac{e^{\theta e^{-\lambda y} (1+x-y)^{-\alpha}}-1}{e^\theta-1}\right)^\beta 1_{x \geq y} +  \left(\frac{e^{\theta e^{-\lambda x} (1+y-x)^{-\alpha}}-1}{e^\theta-1}\right)^\beta 1_{x < y}.
\end{equation*}
Moreover, let us consider the function $\bar{G}_{X,Y}(x,y) = h^{-1}(\bar{F}_{X,Y}(x,y))$, that is given by 
\begin{equation*}
    \bar{G}_{X,Y}(x,y) = e^{-\lambda y}(1+x-y) ^ {-\alpha} 1_{x \geq y} + 
    e^{-\lambda x} (1+y-x) ^ {-\alpha} 1_{x < y}:
\end{equation*}
then its second mixed derivative $g(x,y)$ is given by
\begin{equation*}
    g(x,y)  \dfrac{{\alpha}\cdot\left(-{\lambda}y+\left(x+1\right){\lambda}-{\alpha}-1\right)\mathrm{e}^{-{\lambda}y}}{\left(-y+x+1\right)^{\alpha+2}}1_{x > y}  +\dfrac{{\alpha}\cdot\left(-{\lambda}x+\left(y+1\right){\lambda}-{\alpha}-1\right)\mathrm{e}^{-{\lambda}x}}{\left(-x+y+1\right)^{\alpha+2}}1_{x < y}.
\end{equation*}
Under the conditions on the parameters stated above, we have that, when $(x,y) \rightarrow (0^{+},0^{+})$, $g(x,y) \rightarrow -\alpha(\alpha+1-\lambda) < 0$, so $g$ is not a density function.\\
Furthermore, let us consider the second-mixed derivative of the function $\bar{F}_{X,Y}$ on the set $\{(x,y): x > y \geq 0\}$:
setting $u = e^{-\lambda y}$, $k = x-y$, with $0 < u \leq 1$, $0 < k < \infty$, after some algebra, we have
\begin{equation}
\begin{split}
   & \frac{\partial^2 \bar{F}_{X,Y}}{\partial x \partial y} \left(k-\frac{\log(u)}{\lambda},-\frac{\log(u)}{\lambda}\right) = \\
   & = C(u,k)\{
     \theta [\beta e^{\theta u (1+k)^{-\alpha}}-1] 
      (\lambda + \lambda k -\alpha)u + [e^{\theta u (1+k)^{-\alpha}}-1]  
      (\lambda + \lambda k-\alpha-1)(1+k)^{\alpha}
    \},
    \label{f}
    \end{split}
\end{equation}
where $C(u,k) = \frac{\alpha \beta \theta u}{(e^{\theta}-1)^{\beta}} \ e^{\theta u
     (1+k)^{-\alpha}}
     \ [e^{u \theta (1+k)^{-\alpha}}-1]^{\beta-2} \ (1+k)^{-2\alpha-2} > 0$.
Under the conditions on the parameters stated above,  the function (\ref{f}) is non-negative if 
$k \geq \frac{\alpha+1}{\lambda}-1, \forall u \in (0,1]$, so let us focus on the case in which $0 < k < \frac{\alpha+1}{\lambda}-1$.\\
For this purpose, we define $\rho(k) = \theta (\lambda + \lambda k-\alpha)$, $\gamma(k) = (1+k)^\alpha (\lambda + \lambda k -\alpha -1)$ and $z(k) = \theta (1+k)^{-\alpha}$: since $0 < k < \frac{\alpha+1}{\lambda}-1$ and $\lambda \geq \alpha$, we can easily see that $\rho(k) \geq 0$, $z(k) > 0$ and  $\gamma(k) < 0$.
Basically, we need that the function 
\begin{equation}
    w(u,k) = \rho(k)(\beta e^{u z(k)}-1)u+\gamma(k)(e^{u z(k)}-1)
\end{equation}
is non-negative in the set $\left\{ (u,k): 0 < u \leq 1, 0 < k < \frac{\alpha+1}{\lambda}-1)\right\}$.
But its first partial derivative  
with respect to $u$ is non-negative in that rectangle,
implying that the infimum of $w(u,k)$ on that set is obtained as $u \rightarrow 0^{+}$ and it is equal to $0$.
Similar results hold when $x < y$.\\
So we conclude that $\bar{F}_{X,Y}$ is a bivariate survival function 
but, under the same conditions on the parameters, $\bar{G}_{X,Y}$ is not a bivariate survival function.\\

\end{example}

\vskip 0.3 cm

It is well-known that the only distribution satisfying bivariate standard weak lack-of-memory property with marginals satisfying univariate standard lack-of-memory property is the exponential Marshall-Olkin distribution with survival function given by (\ref{intr4}).

A similar result can be found for the pseudo version of the lack-of-memory properties, as we show in next Proposition.
\begin{proposition}
Let $h:[0,1] \rightarrow [0,1]$ be a generator such that $h^{-1}(x)$ is log-concave.
The only distribution satisfying pseudo bivariate weak lack-of-memory property, with  generator $h$, with marginal survival functions satisfying pseudo univariate lack-of-memory property with the same generator $h$, is 
\begin{equation}
    \begin{split}
     \bar{F}_{X,Y}(x,y)  = \begin{cases}
    \exp_h(\lambda y) \otimes \exp_h(\gamma_1(x-y)), \ x \geq y \\
    \exp_h(\lambda x) \otimes \exp_h(\gamma_2(y-x)), \ x < y \\
    \end{cases}  = \exp_h(\lambda_1 x + \lambda_2 y + \lambda_0 \max(x,y)),
    \label{uni and biv eq}
    \end{split}
\end{equation}
with $0 < \max(\gamma_1,\gamma_2) \leq \lambda \leq \gamma_1 + \gamma_2$ and with $\lambda_1 = \lambda-\gamma_2$, $\lambda_2 = \lambda-\gamma_1$ and $\lambda_0 = \gamma_1+\gamma_2 - \lambda$.\\
\label{prop 3.3.3}
\end{proposition}
\begin{proof}
    This is an immediate consequence of (\ref{PWBLMP_eq}) and of Remark \ref{rm2}.
\end{proof}
Exactly as the classical Marshall and Olkin distribution, the survival function (\ref{uni and biv eq}) can be obtained from the following construction based on a shock model.
In fact, let us consider three random variables $Z_1$,$Z_2$ and $Z_3$ with survival distribution functions
$\bar{F}_{Z_i}(x)= \exp_h(\lambda_i x), i = 1,2,3$ such that
$$P[Z_1 > z_1, Z_2 > z_2, Z_3 > z_3] =  
\exp_h(\lambda_1 z_1) \otimes \exp_h(\lambda_2 z_2) \otimes \exp_h(\lambda_3 z_3),$$
meaning that the associated survival copula is of Archimedean type with generator $\psi(x) = h(e^{-x})$ if $h^{''}(t) \ t^2 + h^{'}(t) \ t \geq 0$ and $h^{'''}(t) \ t^3  + 3 h^{''} (t) t^2 + h^{'}(t) \ t \geq 0, \forall t \in [0,1]$, see Theorem 2 in McNeil and Neslehovà (2009).
Furthermore, let us consider the random variables $X = \min(Z_1,Z_3)$ and $Y = \min(Z_2,Z_3)$.
Then
\begin{equation}
    \begin{split}
         P[X > x, Y > y] & = P[Z_1 > x, Z_2 > y, Z_3 > \max(x,y)]  = \exp_h(\lambda_1 x + \lambda_2 y +\lambda_3 \max(x,y)).
    \end{split}
\end{equation}
Survival functions of type (\ref{uni and biv eq}) represent a particular specification of the family of distributions considered in equation (1) in Mulinacci (2018), when, according to the notation therein, $G(x) = \exp_h(x)$ and $H_i(x) = \lambda_i x, \ i = 1,2,3.$
\vskip 0.3 cm
Unlike the model considered in Proposition \ref{prop 3.3.3}, in the following example we analyse the case in which the bivariate survival function $\bar{F}$ satisfies pseudo weak lack-of-memory property while the marginal distributions satisfy univariate standard lack-of-memory property.
\begin{example}
If $h(x) = \frac{e^{\theta x}-1}{e^{\theta}-1}$ and if $\Bar{F}_i(x) = e^{-\gamma_i x}, \ i = 1,2$, then the function
\begin{equation}
 \begin{split}
     & \Bar{F}(x,y)  = \\
     & = \frac{\exp\{e^{-\lambda y} \log((e^{\theta}-1) e^{-\gamma_1(x-y)}+1)\}-1}{e^{\theta}-1} 1_{x \geq y} + \frac{\exp\{e^{-\lambda x} \log((e^{\theta}-1) e^{-\gamma_2(y-x)}+1)\}-1}{e^{\theta}-1} 1_{x < y}
    \end{split}
    \label{survival example 2}
\end{equation}
is a survival function satisfying (\ref{pseudo weak equation}) if $$
\max(\gamma_1,\gamma_2) \cdot \max\left(1,\frac{e^{\theta}-1}{\theta e^\theta}\right) \leq \lambda \leq 
\frac{(\gamma_1+\gamma_2)(e^{\theta}-1)}{\theta e^\theta}.
$$
Using the conditional distribution method, see Nelsen (2006), we generate data from this distribution with  parameters $\gamma_1 = 0.5$, $\gamma_2 = 0.6$ and $\lambda = 0.645$.
In Figure \ref{fig:three graphs}, we show the scatterplots from (\ref{survival example 2}) for three different values of the parameter $\theta$: we can see that dependence decreases as $\theta$ increases.

\end{example}

\section{Properties and Characterizations of  Pseudo Weak Case}
In the literature, many characterizations of distributions satisfying standard lack-of-memory property have been provided (see, among the others, Kulkarni, 2006).
We report here the characterization provided in section 3 in Ghurye and Marshall (1984) in the bi-dimensional case.
\begin{theorem}
    The survival distribution $\Bar{F}_{X,Y}$ of the vector $(X,Y)$ satisfies standard weak lack-of-memory property if and only if there exist random variables $U$ and $(W_1,W_2)$ such that
    \begin{enumerate}
        \item $(X,Y) = (U,U)  + (W_1,W_2)$;
        \item $U$ and $(W_1,W_2)$ are independent;
        \item $P(\min(W_1,W_2) = 0) = 1$;
        \item $U$ has an exponential distribution.
    \end{enumerate}
\end{theorem}
The authors actually show that $U = \min(X,Y)$;
moreover, they prove that the random variables $U = \min(X,Y)$ and $N = \mathbf{1}_{\{X > Y\}}-\mathbf{1}_{\{X < Y\}}$ are independent.
\vskip 0.2 cm
In the same spirit, as for the case of pseudo weak lack-of-memory property, we get the following result.
\begin{proposition}
    Let $\Bar{F}_{X,Y}$ be a bivariate survival function satisfying pseudo weak lack-of-memory property with marginal survival functions $\bar{F}_1$ and $\bar{F}_2$;
    moreover, let $(W_1,W_2) = (X - U,Y - U)$, where $U = min(X,Y)$.
    Then:
    \begin{enumerate}
        \item $P(\min(W_1,W_2) = 0) = 1$;
        \item $U$ has a pseudo exponential distribution with parameter $\lambda$;
        \item $N = \mathbf{1}_{X > Y}-\mathbf{1}_{X < Y}$ and $U$ are independent;
        \item The joint distribution of $U$ and the vector $(W_1,W_2)$ is given by
        \begin{equation*}
        P(U \geq u, \Bar{W} \geq \Bar{w})  = \begin{cases}
            \exp_h(\lambda u), \ w_1 \leq 0, w_2 \leq 0 \\
            \exp_h(\lambda u) \otimes \bar{F}_1(w_1) \left(1-\frac{r_1(w_1)}{\lambda}\right), \ w_1 > 0, w_2 \leq 0 \\
            \exp_h(\lambda u) \otimes \bar{F}_2(w_2) \left(1-\frac{r_2(w_2)}{\lambda}\right), \ w_1 \leq 0, w_2 > 0 \\
            0, \ w_1 > 0, w_2 > 0
        \end{cases},
        \end{equation*}
        where $r_i$ is the hazard rate of the survival distribution $\Bar{G}_i, \ i = 1,2$.
    \end{enumerate}
\end{proposition}
\begin{proof}
    $\textit{1.}$ holds true for the same reasons given in Ghurye and Marshall (1984).\\
    For $\textit{2.}$, since $\Bar{F}_{X,Y}$ satisfies pseudo weak lack-of-memory property,
    $$\Bar{F}_U(u) = P(X \geq u, Y \geq u) = \Bar{F}_{X,Y}(u,u) = h(e^{-\lambda u}), u \geq 0.$$
    Regarding $\textit{3.}$, let us start with the case in which $N = 1$: then, by (\ref{PWBLMP_eq}),
    \begin{equation}
        \begin{split}
            P(N = 1, U \geq u)  = & P(X > Y \geq u)  = \int_{u}^{\infty} \int_{u}^{x} \frac{\partial^2 \bar{F}_{X,Y}}{\partial x \partial y} \ dy \ dx = \\
            & = \int_{u}^{\infty} -h^{'}(e^{-\lambda y} \Bar{G}_1(x-y)) e^{-\lambda y} 
            g_1(x-y) \big{|}_{y=u}^{x} \ dx= 
            h(e^{-\lambda u}) \left(1-\frac{g_1(0)}{\lambda}\right).
            \label{X > Y > u}
        \end{split}
    \end{equation}
    Using equation (\ref{condition1}),
    $P(N = 1) =  \left(1-\frac{g_1(0)}{\lambda}\right),$
    so (\ref{X > Y > u}) rewrites  
    $$P(N = 1, U \geq u) = P(U \geq u) \ P(N = 1).$$
    Similar results hold for $N = -1$.
    In the case in which $N = 0$,
    $$P(U \geq u, N = 0) = P(X = Y \geq u) = h(e^{-\lambda u}) \left(\frac{g_1(0)+g_2(0)}{\lambda}-1\right),$$ thanks to (\ref{eq prop2}); the conclusion follows taking into account that $P(N = 0) = \left(\frac{g_1(0)+g_2(0)}{\lambda}-1\right)$ by  (\ref{condition3}). \\
    As for $\textit{4.}$, it is trivial to show that $\Bar{F}_{(U,W_1,W_2)}(u,w_1,w_2) = h(e^{-\lambda u})$ if $w_1,w_2 \leq 0$ and that 
    $\Bar{F}_{(U,W_1,W_2)}(u,w_1,w_2) = 0$ if $w_1,w_2 > 0$: for this reason, we focus on the cases in which $w_1 > 0$ and $w_2 \leq 0$ and $w_1 \leq 0$ and $w_2 > 0$.\\
    For $w_1 > 0$ and $w_2 \leq 0$, we have that:
    \begin{equation}
        \begin{split}
               \Bar{F}_{(U,W_1,W_2)}(u,w_1,w_2) & =  P(X \geq u, Y \geq u, X - Y \geq w_1) = P( X \geq w_1 + u, u \leq Y \leq X - w_1) = \\
              & = \int_{u+w_1}^{\infty} \int_{u}^{x-w_1} \frac{\partial^2  \Bar{F}}{\partial x \partial y} dy \ dx  = \int_{u+w_1}^{\infty} \left(-h^{'}(e^{-\lambda y} \Bar{G}_1(x-y)) g_1(x-y) e^{-\lambda y} \big{|}_{y=u}^{x-w_1} \right) dx = \\
             & = h(e^{-\lambda u} \Bar{G}_1(w_1)) \left(1-\frac{g_1(w_1)}{\lambda \Bar{G}_1(w_1)}\right)  = 
             \exp_h(\lambda u) \otimes \Bar{F}_1(w_1) \left(1-\frac{r_1(w_1)}{\lambda}\right).
             \label{U W}
        \end{split}
    \end{equation}
   \\
    Similarly, if $w_2 > 0$ and $w_1 \leq 0$, we have 
    \begin{equation}
     \Bar{F}_{(U,W_1,W_2)}(u,w_1,w_2)   
     = \exp_h(\lambda u) \otimes \Bar{F}_2(w_2) \left(1-\frac{r_2(w_2)}{\lambda}\right).
    \end{equation}
\end{proof}
In Theorem 8.1 in Block and Basu (1974), the following alternative characterization of the vector $(X,Y)$ possessing standard weak lack-of-memory property, in the non-singular case, is given.
\begin{theorem}
    Let (X, Y) have a non-negative bivariate
 distribution which is absolutely continuous with marginal cumulative distribution functions $F_1$ and $F_2$. Then $(X,Y)$ satisfies standard weak lack-of-memory property if and only if there exist random variables $U = \min(X, Y)$ and
 $V = X - Y$ such that
 \begin{enumerate}
     \item $U$ and $V$ are independent
 \item $U$ has an exponential distribution with parameter $\theta > 0$;
 \item \begin{equation*}
     P(V < t) = \begin{cases}
         F_1(t)+\frac{f_1(t)}{\theta}, \ t \geq 0 \\
         1-F_2(-t)-\frac{f_2(-t)}{\theta}, t < 0
     \end{cases}.
 \end{equation*}
 \end{enumerate}
\end{theorem}

In the next Proposition, we extend their characterization to bivariate survival distributions satisfying the pseudo weak lack-of-memory property without assuming any conditions on the singularity.
\begin{proposition}
    Let $F_{X,Y}$ be the bivariate survival function of the vector $(X,Y)$ with marginal survival functions $\bar{F}_i = h(\Bar{G}_i)$, $i = 1,2$.
    Then $\Bar{F}_{X,Y}$ satisfies pseudo weak lack-of-memory property if and only if
        \begin{equation}
          P(U \geq u, V \geq v) = \begin{cases}
             \exp_h(\lambda u) +  h(e^{-\lambda u} \Bar{G}_2(-v)) \left(\frac{ g_2(-v)}{\lambda \Bar{G}_2(-v)}-1\right), u \geq 0, v \leq 0 \\
             h(e^{-\lambda u} \Bar{G}_1(v)) \left(1-\frac{g_1(v)}{\lambda \Bar{G}_1(v)}\right), u \geq 0, v > 0
        \end{cases},
            \label{Surv_VU}
        \end{equation}
        where $U = \min(X,Y)$ and $V = X-Y$.
\end{proposition}
\begin{proof}
    Let us start with the case $v < 0$. We have:
    \begin{equation*}
         P( U \geq u, V \geq v) 
         = P(u \leq Y < X) + P(X = Y \geq u) + P(X \geq u, X < Y \leq X-v).
    \end{equation*}
    The first probability is already known from equation (\ref{X > Y > u}), while the second one is given by (\ref{eq prop2}), so we need to compute only the last one.
    We get
    \begin{equation*}
    \begin{split}
          P(X \geq u, X < Y \leq X-v)   & = \int_{u}^{\infty} \int_{x}^{x-v} \frac{\partial^2 \bar{F}_{X,Y}(x,y)}{\partial x \partial y} \ dy \ dx = \\
        & = \int_{u}^{\infty} h^{'}(e^{-\lambda x} \Bar{G}_2(y-x)) (-\lambda e^{-\lambda x} \bar{G}_2(y-x) +  e^{-\lambda x} g_2(y-x)) \big{|}_{y = x}^{x-v}  \ dx  = \\
        & = h(e^{-\lambda u} \Bar{G}_2(-v)) \left(\frac{ g_2(-v)}{\lambda \Bar{G}_2(-v)}-1\right) + h(e^{-\lambda u}) \left(1-\frac{ g_2(0)}{\lambda}\right).
    \end{split}
    \end{equation*}
Overall,
\begin{equation*}
    \begin{split}
    & P(U \geq u, V \geq v) =  h(e^{-\lambda u}) +  h(e^{-\lambda u} \Bar{G}_2(-v)) \left(\frac{ g_2(-v)}{\lambda \Bar{G}_2(-v)}-1\right).
    \end{split}
\end{equation*}
If $v > 0$, by equation (\ref{U W}),
\begin{equation*}
    P(U \geq u, V \geq v) =
    h(e^{-\lambda u} \Bar{G}_1(v)) \left(1-\frac{g_1(v)}{\lambda \Bar{G}_1(v)}\right).
\end{equation*}
Conversely, let us suppose that (\ref{Surv_VU}) holds true.
If $0 \leq x \leq y$, we have that
\begin{equation*}
    \begin{split}
         & \Bar{F}_{X,Y}(x,y)  = \\
         & = P(X > x, Y > y, X \geq Y) + 
        P( X > x, Y > y, X < Y) = \\
        & = P(U + V > x, U > y, V \geq 0) + P(U > x, U -V > y, V < 0) = \\
        & = P(U > y, V \geq 0) + P(U > x, U -V > y, V < 0) = 
        \\
        & = P(U > y, V \geq 0) + P( x < U < y, \ x-y < V < U-y) + P(U > y, x-y < V < 0) + P(U > x, V \leq x-y).
    \end{split}
\end{equation*}
From (\ref{Surv_VU}), $P(U \geq y, V \geq 0) = h(e^{-\lambda u}) \frac{g_2(0)}{\lambda}$;
moreover,
\begin{equation*}
    \begin{split}
     P(U > y, x-y < V < 0)  & = \int_{y}^{\infty}  -\lambda e^{-\lambda u} h^{'}(e^{-\lambda u} \Bar{G}_2(-v)) \ \Bar{G}_2(-v) 
    \left(\frac{g_2(-v)}{\lambda \Bar{G}_2(-v)}-1\right) \big{|}_{v = x-y}^0 \ du = \\
    & = h(e^{-\lambda y} \Bar{G}_2(y-x)) \left(\frac{g_2(y-x)}{\lambda \Bar{G}_2(y-x)}-1\right)-h(e^{-\lambda y}) \left(\frac{g_2(0)}{\lambda}-1\right);
    \end{split}
\end{equation*}
similarly, 
\begin{equation*}
\begin{split}
 & P(x < U < y , \ x-y < V < U-y)  = \\
& = h(e^{-\lambda x} \Bar{G}_2(y-x)) \frac{g_2(y-x)}{\lambda \Bar{G}_2(y-x)} - h(e^{-\lambda y})  -h(e^{-\lambda y} \Bar{G}_2(y-x)) \left(\frac{g_2(y-x)}{\lambda \Bar{G}_2(y-x)}-1\right).
\end{split}
\end{equation*}
Finally, 
\begin{equation*}
    \begin{split}
        & P(U > x, V \leq x-y) =
         h(e^{-\lambda x} \Bar{G}_2(y-x)) \left(1-\frac{g_2(y-x)}{\lambda \bar{G}_2(y-x)}\right).
    \end{split}
\end{equation*}
Summing up all the probabilities above, we get that 
$$\Bar{F}_{X,Y}(x,y) = h(e^{-\lambda x} \Bar{G}_2(y-x)), \ 0 \leq x \leq y;$$ similar results hold for $x > y$.
\end{proof}
\newpage
\section{Dependence Structure of Pseudo Weak Distribution}
In this section, we analyse the dependence structure of distributions satisfying (\ref{pseudo weak equation}), showing the impact that the generator $h$ has on the Kendall distribution function and on tail dependence coefficients.
\subsection{Kendall Distribution Function}
We recall that the Kendall Distribution function of a random vector $(X,Y)$ with survival distribution $\bar{F}_{X,Y}$ is defined as
$$K(t) = P(\bar{F}_{X,Y}(X,Y) \leq t), \ t \in [0,1],$$
see, among the wide literature, Genest and Rivest (2001) and Joe (2014). \\
As proved in next Proposition, the Kendall Distribution Function associated to the survival function of type (\ref{PWBLMP_eq}) can be expressed in terms of $\bar{G}_i, i = 1,2$ and in terms of $h$.
\begin{proposition}
Let
$K$
be the Kendall distribution function associated to the survival function (\ref{PWBLMP_eq}).
Let us assume that $h$ is differentiable and that $\bar G_i=h^{-1}\left (\bar F_i\right )$ admits a density $g_i$ for $i=1,2$. Then
\begin{equation}
K(s)=
s-H(h^{-1}(s))
\label{Kendall_equation}
\end{equation}
where
$$H(v)=h^\prime\left (v\right)v\left[2\ln\left(v\right)+\frac 1\lambda \left(J_1(v)+J_2(v)\right)\right],$$
with
\begin{equation}
J_i(v)=\int_0^{\bar G_i^{-1}\left (v\right )}\frac{g_i^2(z)}{\bar G_i^2(z)}dz,\,i=1,2.
\label{J}
\end{equation}
\label{Kendall_prop}
\end{proposition}
\begin{proof}
Since $\bar F_{X,Y}(x,y)=h\left (\bar G_{X,Y}(x,y\right ))$, we have
$$\mathbb P\left (\bar F_{X,Y}(X,Y)\leq s\right )=\mathbb P\left (\bar G_{X,Y}(X,Y)\leq h^{-1}(s)\right ).$$
Let $z_s=-\frac 1\lambda\ln \left(h^{-1}(s)\right)$ be the solution of $\bar G_{X,Y}(x,x)=h^{-1}(s)$ and let $D_1$ and $D_2$ be defined as
$$D_1=\left\{(x,y):0\leq y\leq z_s,\,z_s\leq x\leq y+\bar G_1^{-1}\left (h^{-1}(s)e^{\lambda y}\right )\right\}$$
and
$$D_2=\left\{(x,y):0\leq x\leq z_s,\,z_s\leq y\leq x+\bar G_2^{-1}\left (h^{-1}(s)e^{\lambda y}\right )\right\}.$$
Then
$$\begin{aligned}K(s)&=\bar F_1(z_s)+\bar F_2(z_s)-\mathbb P\left( (X,Y)\in D_1\right )-\mathbb P\left( (X,Y)\in D_2\right )-\bar F_{X,Y}(z_s,z_s)=\\
&=\bar F_1(z_s)+\bar F_2(z_s)-\mathbb P\left( (X,Y)\in D_1\right )-\mathbb P\left( (X,Y)\in D_2\right )-s
.\end{aligned}$$
Noticing that
$$\mathbb P(X>x\vert Y=y)\ f_2(y)=-\frac{\partial \bar F_{X,Y}(x,y)}{\partial y}=-h^\prime (\bar G_{X,Y}(x,y))\frac{\partial \bar G_{X,Y}(x,y)}{\partial y}$$ and evaluating it in 
$x=y+\bar G_1^{-1}\left (h^{-1}(s)e^{\lambda y}\right )$,
we have
$$\begin{aligned}
\mathbb P((X,Y)\in D_1)&=\\
&= \int_0^{z_s}\left [\mathbb P\left (X>z_s\vert Y=y\right )-
\mathbb P\left (X>y+\bar G_1^{-1}\left (h^{-1}(s)e^{\lambda y}\right )\vert Y=y\right )\right ]f_2(y)\,dy=\\
&=\int_0^{z_s} \left(
h^\prime(h^{-1}(s))\left [-\lambda h^{-1}(s)+e^{-\lambda y}g_1\left (\bar G_1^{-1}\left (h^{-1}(s)e^{\lambda y}\right )\right )\right] + \frac{\partial \bar F_{X,Y}(z_s,y)}{\partial y} \right) dy = \\ & = -\bar F_{X,Y}(z_s,z_s)+\bar F_1(z_s)
-\lambda h^\prime(h^{-1}(s))h^{-1}(s)z_s + h^\prime(h^{-1}(s))
\int_0^{z_s}e^{-\lambda y}g_1\left (\bar G_1^{-1}\left (h^{-1}(s)e^{\lambda y}\right )\right )\,dy=\\
&=-s+\bar F_1(z_s)
+ h^\prime(h^{-1}(s)) h^{-1}(s)\ln\left (h^{-1}(s)\right )  +\frac 1\lambda h^\prime(h^{-1}(s))h^{-1}(s)
\int_0^{\bar G_1^{-1}\left (h^{-1}(s)\right )}\frac{g_i^2(z)}{\bar G_1^2(z)}
\,dy
\end{aligned},$$
where, in the last integral, we have substituted $z=\bar G_1^{-1}\left (h^{-1}(s)e^{\lambda y}\right )$.

The probability $\mathbb P((X,Y)\in D_2)$ can be obtained similarly.
\end{proof}
\vskip 0.2 cm
Since the standard weak lack-of-memory property can be obtained from the pseudo one when $h = id$, Proposition \ref{Kendall_prop} allows to recover the Kendall Distribution Function also in the standard setting, id est
\begin{equation}
    K(t) = t \left(1-2\log(t) - \frac{J_1(t)+J_2(t)}{\lambda}\right),
    \label{K_id}
\end{equation}
where $J_i, \ i = 1,2$ are given by (\ref{J}). \\
We now give some examples in which the expressions for $J_i, \ i = 1,2$ can be determined explicitly.
\begin{example}
    If $\Bar{G}_i(x) = e^{-\alpha_i x}, \ i = 1,2$, then $J_i(x) = -\alpha_i \log(x), \ i = 1,2$ and 
    $$K(t) = t-h^{'}(h^{-1}(t)) h^{-1}(t) \left[ 2 \log(h^{-1}(t)) - \frac{\alpha_1+\alpha_2}{\lambda} \log(h^{-1}(t))\right].$$ If $h(x) = x$, with $\max(\alpha_1,\alpha_2)  \leq \lambda \leq \alpha_1+\alpha_2$, we get the well-known Marshall-Olkin distribution, see Marshall and Olkin (1967). 
    Using (\ref{K_id}), we recover the well-known expression for the Kendall distribution function of the Marshall-Olkin distribution:
    \begin{equation*}
        \begin{split}
            & K(t) = t \left(1-\log(t) \left(2-\frac{\alpha_1+\alpha_2}{\lambda}\right)\right).
        \end{split}
    \end{equation*}
    If $h(x) = e^{-\gamma(x^{-1}-1)}$, with $\max(\alpha_1,\alpha_2)  \leq \lambda \leq \alpha_1+\alpha_2$ and $\gamma \geq 1$, from (\ref{PWBLMP_eq}) we recover
    \begin{equation}
        \bar{F}(x,y) = \begin{cases}
            e^{-\gamma \left[\left(
            1+\frac{\alpha_1(x-y)}{\gamma}\right) e^{\lambda y}-1\right]}, \ x \geq y \\
            e^{-\gamma \left[\left(
            1+\frac{\alpha_2(y-x)}{\gamma}\right) e^{\lambda x}-1\right]}, \ x < y \\
        \end{cases}.
        \label{F1}
    \end{equation}
    The Kendall Distribution Function associated to the latter is given by
    \begin{equation*}
        \begin{split}
            & K(t) = t \left(1-(\gamma-\log(t)) \left(\frac{\alpha_1+\alpha_2-2\lambda}{\lambda} \log\left(1-\frac{\log(t)}{\gamma}\right)\right)\right).
        \end{split}
    \end{equation*}
    Plot of the Kendall distribution function for different values of $\gamma$ are given in Figure \ref{Figure13}, with $\alpha_1 = 2$, $\alpha_2 = 3$ and $\lambda = 4.5$: we can see that distribution (\ref{F1}) can be used to model positive and negative dependence.
 \newline
    If $h(x) = \frac{e^{\gamma x}-1}{e^{\gamma}-1}$, with $\max(\alpha_1,\alpha_2)  \leq \lambda \leq \alpha_1+\alpha_2$ and $\gamma > 0$, from (\ref{PWBLMP_eq}) we recover
    \begin{equation}
        \bar{F}(x,y) = \begin{cases}
            \frac{\exp(\gamma e^{-\lambda y -\alpha_1(x-y)})-1}{e^{\gamma}-1}, \ x \geq y \\
            \frac{\exp(\gamma e^{-\lambda y -\alpha_2(y-x)})-1}{e^{\gamma}-1}, \ x < y \\
        \end{cases}.
        \label{F13}
    \end{equation}
    The Kendall Distribution Function associated to the latter is given by
    \begin{equation*}
        \begin{split}
            & K(t) = t+(1+t(e^\gamma-1)) \log(1+t(e^\gamma-1)) \frac{\alpha_1+\alpha_2-2\lambda} {\lambda (e^\gamma-1)} \left(\log(\log(1+t(e^\gamma-1))-\log(\gamma)\right)
        \end{split}
    \end{equation*}
    Standard lack-of-memory property can be recovered when $\gamma \rightarrow 0^{+}$.
    Plot of the Kendall distribution function for different values of $\gamma$ are given in Figure \ref{Figure15}, with $\alpha_1 = 2$, $\alpha_2 = 3$ and $\lambda = 4.5$: we can see that dependence is positive and it increases as $\gamma$ increases.
 \newline
\end{example}

\begin{example}
    If $\Bar{G}_i(x) = (1+x)^{-\alpha_i}, \ i = 1,2$, then $J_i(x) = \alpha_i^2 (1-x^{\frac{1}{\alpha_i}}), \ i = 1,2$ and 
    \begin{equation}
    K(t) = t-h^{'}(h^{-1}(t)) h^{-1}(t) \left[ 2 \log(h^{-1}(t)) + \frac{\alpha_1^2(1-(h^{-1}(t))^{\frac{1}{\alpha_1}})+\alpha_2^2(1-(h^{-1}(t))^{\frac{1}{\alpha_2}})}{\lambda} \right].
    \label{K_Pareto}
    \end{equation}
    If $h(x) = x$, with $\max(\alpha_1,\alpha_2) +1 \leq \lambda \leq \alpha_1+\alpha_2$, (\ref{PWBLMP_eq}) becomes
    \begin{equation*}
        \bar{F}(x,y) = \begin{cases}
            e^{-\lambda y} (1+x-y)^{-\alpha_1}, x \geq y \\
            e^{-\lambda x} (1+y-x)^{-\alpha_2}, x < y \\
        \end{cases}.
    \end{equation*}
    The Kendall Distribution Function associated to the latter is given by
    \begin{equation*}
        \begin{split}
            & K(t) = t \left(1-2 \log(t) - \frac{\alpha_1^2(1-t^{\frac{1}{\alpha_1}})+ \alpha_2^2 (1-t^{\frac{1}{\alpha_2}})}{\lambda}\right).
        \end{split}
    \end{equation*}
    If $h(x) = e^{-\gamma(x^{-1}-1)}$, with $\max(\alpha_1,\alpha_2) +1  \leq \lambda \leq \alpha_1+\alpha_2$ and $\gamma \geq 2$, from (\ref{PWBLMP_eq}) we recover
    \begin{equation}
    \Bar{F}(x,y) = 
    \begin{cases}
        e^{-\gamma \left[\left(1+\frac{\alpha_1}{\gamma} \log(1+x-y\right) e^{\lambda y}-1\right]}, \ x \geq y \\
        e^{-\gamma \left[\left(1+\frac{\alpha_2}{\gamma} \log(1+y-x\right) e^{\lambda x}-1\right]}, \ x < y \\
    \end{cases}.
    \label{F2}
    \end{equation}
    The Kendall Distribution Function associated to the latter is given by
    \begin{equation*}
        \begin{split}
            & K(t) = t\cdot\left(1-\left({\gamma}-\ln\left(t\right)\right)\left(\dfrac{{\alpha}_1^2\cdot\left(1-\left(1-\frac{\ln\left(t\right)}{{\gamma}}\right)^{-\frac{1}{{\alpha}_1}}\right)}{{\lambda}}+\dfrac{{\alpha}_2^2\cdot\left(1-\left(1-\frac{\ln\left(t\right)}{{\gamma}}\right)^{-\frac{1}{{\alpha}_2}}\right)}{{\lambda}}-2\ln\left(1-\dfrac{\ln\left(t\right)}{{\gamma}}\right)\right)\right)
        \end{split}.
    \end{equation*}
In Figure \ref{Figure 52b}, we plot the Kendall distribution function for different values of $\gamma$, with $\alpha_1 = 10$, $\alpha_2 = 18$, $\lambda = 25$: we can see that dependence slightly increases as $\gamma$ increases.
  
\end{example}
\begin{example}
If $\bar{G}_i(x) = (1+\xi_i(e^{\beta_i x}-1))^{-1}, \ i = 1,2$, then $J_i(x) = (\xi_i-1) \beta_i (1-x)-\beta_i \log(x)$ and 
\begin{equation*}
\begin{split}
& K(t) = t-h^{'}(h^{-1}(t)) h^{-1}(t) \cdot \\
& \cdot \left[ 2 \log(h^{-1}(t)) + \frac{(\xi_1-1) \beta_1 (1-h^{-1}(t))-\beta_1 \log(h^{-1}(t))+(\xi_2-1) \beta_2 (1-h^{-1}(t))-\beta_2 \log(h^{-1}(t))}{\lambda} \right]. 
\end{split}
\end{equation*}
If $h(x) = x$, with $\max(\xi_1,\xi_2) \leq 1$ and $\max(\beta_1,\beta_2) \leq \lambda \leq \xi_1 \beta_1+ \xi_2 \beta_2$, (\ref{PWBLMP_eq}) becomes 
\begin{equation}
\bar{F}(x,y) = 
\begin{cases}
    e^{-\lambda y} (1+\xi_1(e^{\beta_1 (x-y)}-1))^{-1}, \ x \geq y \\
    e^{-\lambda x} (1+\xi_2(e^{\beta_2 (y-x)}-1))^{-1}, \ x < y \\
\end{cases}.
\label{F53}
\end{equation}
    The Kendall Distribution Function associated to the latter is given by
    \begin{equation*}
             K(t) 
             = t\left(1+\frac{\log(t)}{\lambda}(\beta_1+\beta_2-2\lambda) + \frac{1-t}{ \lambda} (\beta_1(1-\xi_1)+\beta_2(1-\xi_2))\right).
    \end{equation*}
If $h(x) = e^{-\gamma(x^{-1}-1)}$, with $\lambda \geq \max(\beta_1,\beta_2)$, 
$\beta_1 \ (1-\gamma \xi_1) \geq \lambda \ (1-\gamma)$, 
$\beta_2 \ (1-\gamma \xi_2) \geq  \lambda\ (1-\gamma)$ and 
$\beta_1 \xi_1+\beta_2 \xi_2-\lambda \geq 0$, we get the same distribution obtained in Marshall and Olkin (2015). Then it is possible to show that
    \begin{equation*}
    \begin{split}
     K(t) 
     = t \left(1-\frac{1}{\lambda }\left(\frac{\ln\left(t\right)\left({\beta}_2\cdot\left({\xi}_2-1\right)+{\beta}_1\cdot\left({\xi}_1-1\right)\right)}{\ln\left(t\right)-{\gamma}}+\left({\beta}_2+{\beta}_1-2\lambda\right)\ln\left(1-\frac{\ln\left(t\right)}{{\gamma}}\right)\right)\left({\gamma}-\ln\left(t\right)\right)\right)
    \end{split}.
    \end{equation*}
\end{example}
\subsection{Upper and Lower Tail Dependence Coefficients}
Given a random vector $(X,Y)$ with copula $C$ and marginal cumulative distribution functions $F_X$ and $F_Y$, we recall that the upper tail dependence coefficient $\lambda_U$ is
\begin{equation*}
    \lambda_U = \lim_{u \rightarrow 1^{-}} P[F_X(X) > u | \ F_Y(Y) > u ] = \lim_{u \rightarrow 1^{-}} \frac{1-2u+C(u,u)}{1-u};
\end{equation*}
analogously, the lower tail dependence coefficient $\lambda_L$ is given by
\begin{equation*}
    \lambda_L = \lim_{u \rightarrow 0^{+}} P[F_X(X) < u | \ F_Y(Y) < u ] = \lim_{u \rightarrow 0^{+}} \frac{C(u,u)}{u}.
\end{equation*}
Thanks to (\ref{distortion copula}), the survival copula associated to the survival distribution satisfying pseudo weak lack-of-memory property is $C^{F}(u,v) = h(C^{G}(h^{-1}(u),h^{-1}(v))$.
The following Propositions in Durante et al. (2010) applies to this case.
\begin{proposition}
    Let $C$ be a copula with finite lower tail dependence coefficient $\lambda_L(C)$. Moreover, let $\psi$ be an isomorphism of $[0,1]$ such that 
    $C_{\psi}(u,v) = 
    \psi \{C[\psi^{-1}(u),\psi^{-1}(v)]\}$ is again a copula. Then, if
    \begin{equation*}
        \lim_{t \rightarrow 0^{+}} \frac{\psi(t)}{t^{\alpha}} = b \in (0, + \infty)
    \end{equation*}
    for some $\alpha > 0$, then $\lambda_L(C_{\psi}) = (\lambda_L(C))^{\alpha}$.
    \label{lambda_L_prop}
\end{proposition}
\begin{proposition}
    Let $C$ be a copula with finite upper tail dependence coefficient $\lambda_U(C)$. Moreover, let $\psi$ be an isomorphism of $[0,1]$ such that 
    $C_{\psi}(u,v) = 
    \psi \{C[\psi^{-1}(u),\psi^{-1}(v)]\}$ is again a copula. Then, if
    \begin{equation*}
        \lim_{t \rightarrow 1^{-}} \frac{1-\psi(t)}{(1-t)^{\alpha}} = b \in (0, + \infty)
    \end{equation*}
    for some $\alpha > 0$, then $\lambda_U(C_{\psi}) = 2-(2-\lambda_U(C))^{\alpha}$.
    \label{lambda_U_prop}
\end{proposition}
In the particular case of identical marginal survival functions, that is when the copula $\Bar{C}^G$ is of the form
\begin{equation}
\bar{C^G}(u,v) 
 = e^{-\lambda \bar{G}^{-1}(v)} \bar{G}(\bar{G}^{-1}(u)-\bar{G}^{-1}(v)) 1_{u < v} + 
e^{-\lambda \bar{G}^{-1}(u)} \bar{G}(\bar{G}^{-1}(v)-\bar{G}^{-1}(u)) 1_{u \geq v},
\label{MO copula standard eq}
\end{equation}
the following Proposition holds true.
\begin{proposition}
    Let $(X,Y)$ be a random vector with survival copula $\Bar{C^G}$ of type (\ref{MO copula standard eq}) and let $g(x) = -\Bar{G}^{'}(x), \ x \geq 0$ such that $\lambda \leq 2 g(0)$ and $\frac{\partial \log(g(z))}{\partial z} \geq -\lambda, \ \forall z \geq 0$.
    Then:
    \begin{enumerate}
    \item if $\Bar{G}$ is heavy tailed, id est $\lim_{x \to \infty} \frac{\Bar{G}(x)}{e^{-\theta x}} = + \infty \ \forall \theta > 0$, then $\lambda_L = 0$.
    \item $
    \lambda_U = 2-\frac{\lambda} {g(0)}.
    $ 
\end{enumerate}
    \label{heavy_prop}
\end{proposition}
\begin{proof}
    For \textit{1.}, we have 
    \begin{equation*}
        \lambda_L = \lim_{u \rightarrow 0^+} \frac{\bar{C^G}(u,u)}{u} = 
        \lim_{u \rightarrow 0^{+}} \frac{e^{-\theta \Bar{G}^{-1}(u)}}{u};
    \end{equation*}
    setting $x = \bar{G}^{-1}(u)$, we get
    \begin{equation*}
        \lambda_L = \lim_{x \rightarrow \infty} \frac{e^{-\theta x}}{\Bar{G}(x)} = 0.
    \end{equation*}
Regarding \textit{2.}, we can write
\begin{equation*}
         \lambda_U  = \lim_{u \rightarrow 1^{-}} \frac{1-2u+\Bar{C^G}(u,u)}{1-u} = 2+  \lim_{u \rightarrow 1^{-}} \frac{e^{-\theta \Bar{G}^{-1}(u)}-1}{1-u} 
         = 2+\lim_{x \rightarrow 0^{+}} \frac{e^{-\theta x}-1}{1-\Bar{G}(x)} = 2-\theta \lim_{x \rightarrow 0^{+}} \frac{x}{G(x)},
\end{equation*}
from which the conclusion follows.
\end{proof}
In the case in which the marginal distribution is light-tailed, the value of the lower tail dependence coefficient depends on the functional form of the distribution. \\
\vskip 0.3 cm

Using Propositions \ref{lambda_L_prop}, \ref{lambda_U_prop} and \ref{heavy_prop}, we are able to find the lower and the upper tail dependence coefficients for different choices of the common marginal survival functions and of the generator $h$, as shown in the following Examples.
\begin{example}
    Let $\Bar{F}_i(x) = \exp_h(\mu x)$, $i = 1,2$, $\mu \leq \lambda \leq 2 \mu$. \\
    If $h = id$, we get the standard Marshall Olkin distribution: 
    it is well-known that the lower tail dependence coefficient of this distribution is equal to $0$ if $\mu < \lambda \leq 2 \mu$ and to $1$ if $\lambda = \mu$.
    Similarly, using Proposition \ref{lambda_U_prop}, we recover the well-known upper tail dependence coefficient, that is equal to $2-\frac{\lambda}{\mu}$.\\
    If $h(x) = 1-\left(\frac{\tan(\theta(1-x))}{\tan(\theta)}\right)^{\beta}, -\frac{\pi}{2} < \theta < 0, \ 0 < \beta < 1$, then $h$ is a convex bjection of the unit interval and Proposition \ref{lambda_L_prop}
    is satisfied with $\alpha = 1$ and with $b = \frac{2 \theta \beta}{\sin(2 \theta)}$, implying that $\lambda_L(\Bar{C}^F) = \lambda_L(\Bar{C}^G)$.
    Moreover, Proposition \ref{lambda_U_prop} holds true with parameters $\alpha = \beta$ and $b = \left(\frac{\theta}{\tan(\theta)}\right)^{\beta}$, implying that $\lambda_U(\bar{C}_F) = 2-\left(\frac{\lambda}{\mu}\right)^{\beta} \in [0,1]$.    
\end{example}
\begin{example}
    Let $\Bar{F}_i(x) = h((1+x)^{-\gamma})$, $i = 1,2$, $\gamma > 0$, $\gamma +1 \leq \lambda \leq 2 \gamma$. \\
    If $h = id$, by Proposition \ref{heavy_prop},  $\lambda_L(\Bar{C}^G)  = 0$  and $\lambda_U(\Bar{C}^G) = 2-\frac{\lambda}{\gamma} \in [0,1]$. \newline
    If $h(x) = 1-\left(\frac{e^{\theta (1-x)}-1}{e^\theta-1}\right)^{\beta}, \ \theta < 0, \ 0 < \beta < 1$, then $h$ is a convex bjection of the unit interval $[0,1]$ and  condition given in Proposition \ref{lambda_L_prop} is satisfied with $\alpha = 1$ and with $b = \frac{\theta \beta e^\theta}{e^\theta-1}$, implying that $\lambda_L(\Bar{C}^F) = \lambda_L(\Bar{C}^G) = 0$.
    Moreover, by Proposition (\ref{lambda_U_prop}), setting $\alpha = \beta$ and $b = \frac{ \theta}{(e^\theta-1)^\beta}$, we get $\lambda_U(\bar{C}^F) =  2-\left(\frac{\lambda}{\gamma}\right)^\beta \in [0,1]$.\\
     Using conditional distribution method, we simulate data with parameters $\alpha = 3$, $\lambda = 4.5$ and $\theta = -0.01$ from the survival distribution function
     \begin{equation}
         \Bar{F}_{X,Y}(x,y) 
          = 1-\left(\frac{e^{\theta (1-(1+x-y)^{-\gamma})}-1}{e^{\theta}-1}\right)^\beta 1_{x \geq y} + 1-\left(\frac{e^{\theta (1-(1+y-x)^{-\gamma})}-1}{e^{\theta}-1}\right)^\beta 1_{x < y},
         \label{scatter survival}
     \end{equation} 
     with Kendall distribution function given by 
     \begin{equation*}
     \begin{split}
& K(t) = t -\dfrac{{\beta}{\theta}\cdot\left(1-\frac{\ln\left(\left(\mathrm{e}^{\theta}-1\right)\left(1-t\right)^\frac{1}{{\beta}}+1\right)}{{\theta}}\right) \left(\left(\mathrm{e}^{\theta}-1\right)\left(1-t\right)^\frac{1}{{\beta}}+1\right)\left(1-t\right)^\frac{{\beta}-1}{{\beta}}}{\mathrm{e}^{\theta}-1} \cdot \\
& \cdot \left(\frac{2\gamma^2}{\lambda}\left(1-\left(1-\frac{\ln\left(\left(\mathrm{e}^{\theta}-1\right)\left(1-t\right)^\frac{1}{{\beta}}+1\right)}{{\theta}}\right)^\frac{1}{{\gamma}}\right)+2\ln\left(1-\frac{\ln\left(\left(\mathrm{e}^{\theta}-1\right)\left(1-t\right)^\frac{1}{{\beta}}+1\right)}{{\theta}}\right)\right),
\end{split}
     \end{equation*}
     see (\ref{K_Pareto}).
     The scatterplots and the Kendall distribution function are given in Figures \ref{Figure6}, \ref{Figure7} and \ref{Figure8} for three different values of $\beta$, with $\gamma = 3$, $\lambda = 4.5$ and $\theta = -0.01.$

Notice that dependence decreases as $\beta$ increases: standard lack-of-memory property is obtained when $\beta = 1$ and $\theta \rightarrow 0^{-}$.
\end{example}
\section{Conclusions}
In this paper, we have generalized lack-of-memory properties by substituting into (\ref{intr5}) and (\ref{intr4}) the standard product by the pseudo product $\otimes$: we have discussed the conditions under which the solutions of the obtained functional equations (\ref{intr_pseudo_strong}) and (\ref{intr_pseudo_weak}) are bivariate survival functions.
Then we have focused our attention on the solution of (\ref{intr_pseudo_weak}): we have proved that it may have a singularity along the line $x=y$ and we have analysed how the probability mass is distributed along that line.
Moreover, we have studied some properties of this distribution and we have characterized it in the spirit of Block and Basu (1974).
Finally, we have analysed the dependence structure of the distribution satisfying (\ref{intr_pseudo_weak}) showing that its associated survival copula can be written in terms of the survival copula of the underlying distribution which satisfies the standard weak lack-of-memory property: in particular, we have obtained an expression for the Kendall distribution function in full generality and we have determined analytically the lower and the upper tail dependence coefficients for specific choices of the generator and of the distorted marginal survival functions. \\

\section{Figures}
\begin{figure}[h!]
     \centering
      \begin{tabular}{@{}c@{}}
         \centering
         \includegraphics[width= 5 cm]{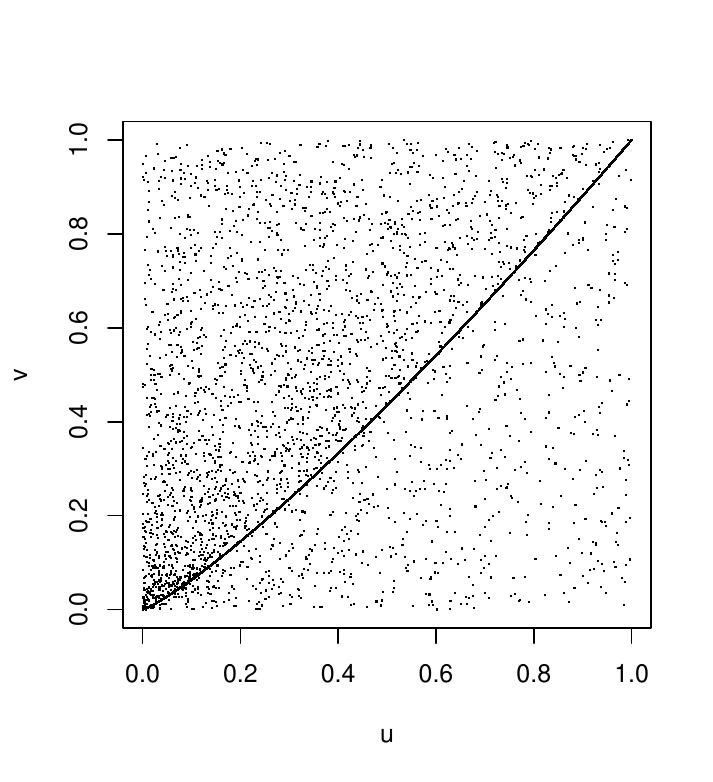}
     \end{tabular}
     \hfill
      \begin{tabular}{@{}c@{}}
         \centering
         \includegraphics[width= 5 cm]{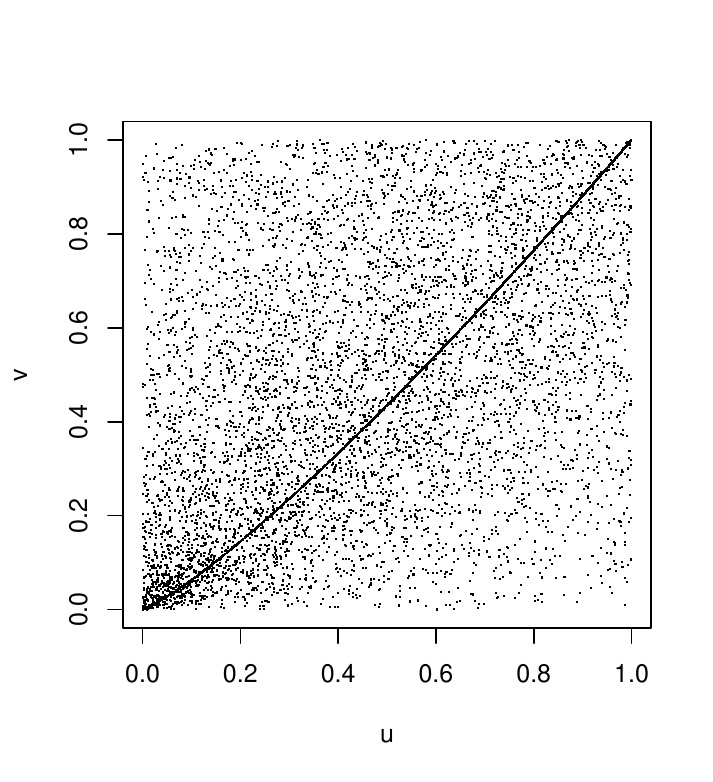}
     \end{tabular}
     \hfill
      \begin{tabular}{@{}c@{}}
         \centering
         \includegraphics[width= 5.5 cm]{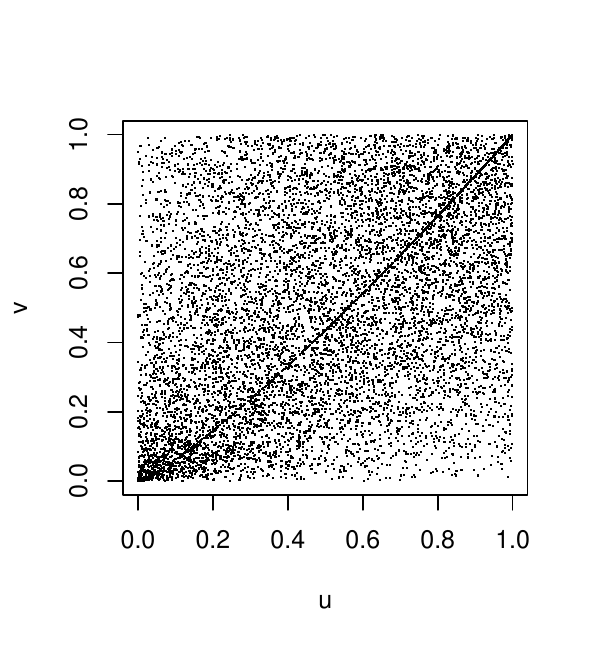}
     \end{tabular}
        \caption{
        Scatterplots from (\ref{survival example 2}).
        Left: $\theta = 0.01$. Center: $\theta = 0.5$. Right: $\theta = 0.99$.}
        \label{fig:three graphs}
\end{figure}
          \begin{figure}[h!]
     \centering
      \begin{tabular}{@{}c@{}}
         \centering
         \includegraphics[width= 5.5 cm]{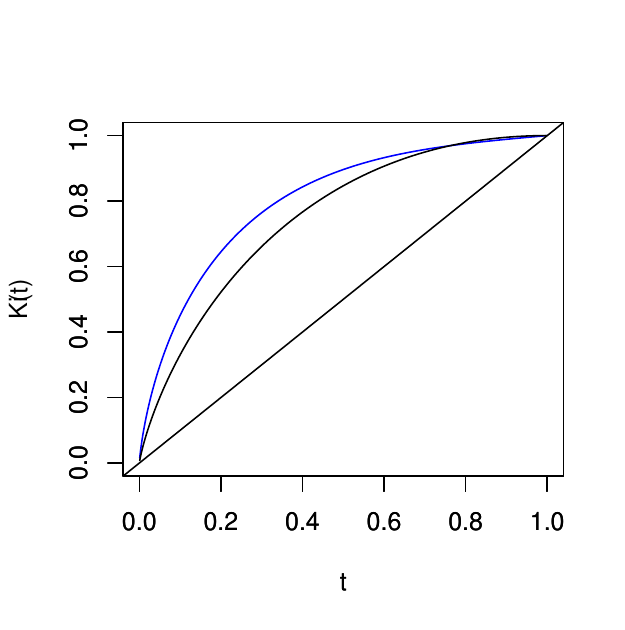}
     \end{tabular}
     \hfill
      \begin{tabular}{@{}c@{}}
         \centering
         \includegraphics[width= 5.5 cm]{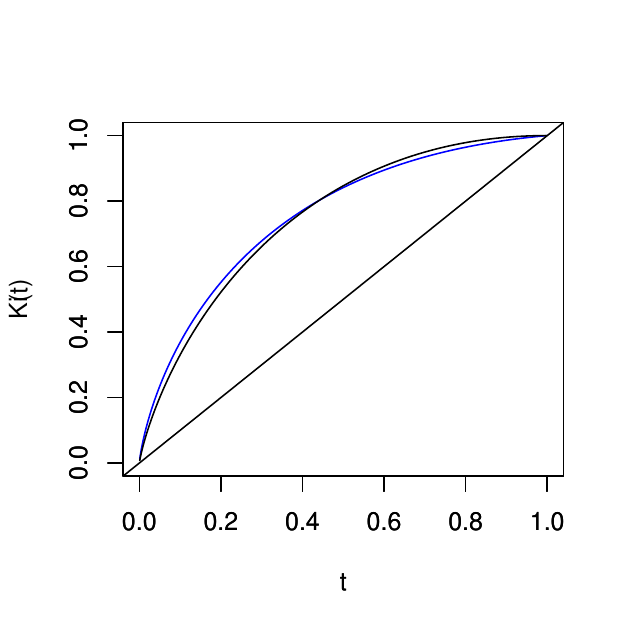}
     \end{tabular}
     \hfill
      \begin{tabular}{@{}c@{}}
         \centering
         \includegraphics[width= 5.5 cm]{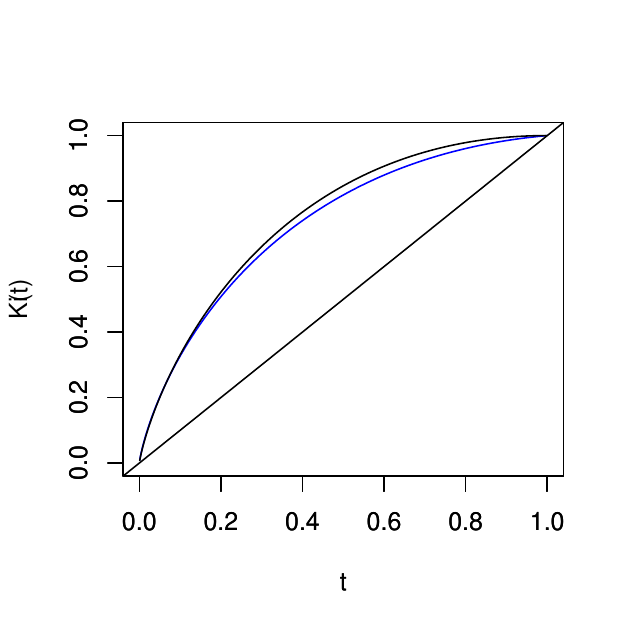}
     \end{tabular}
        \caption{
        Blue Curve: Kendall distribution function of (\ref{F1}).
        Black Curve: Independence Curve.
         Left: $\gamma = 1$; Center: $\gamma = 3$; Right: $\gamma = 5$.}
        \label{Figure13}
\end{figure}
          \begin{figure}[h!]
     \centering
      \begin{tabular}{@{}c@{}}
         \centering
         \includegraphics[width= 5.5 cm]{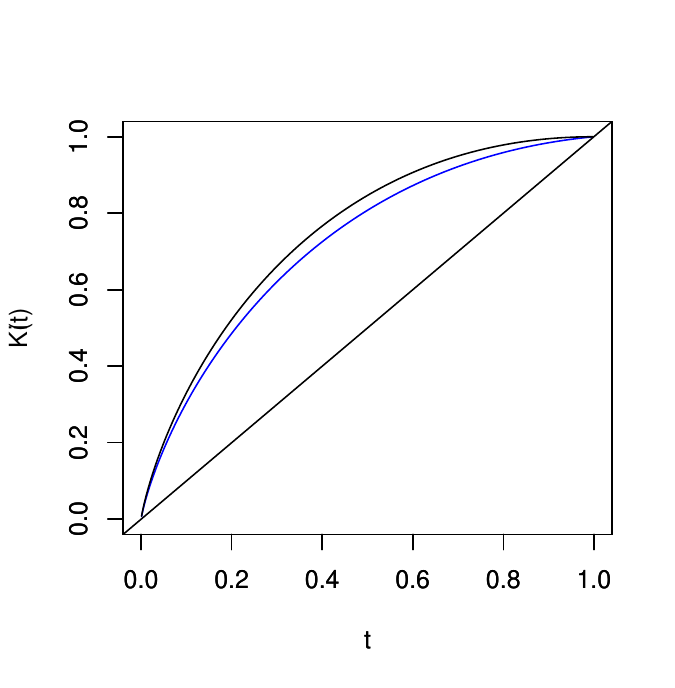}
     \end{tabular}
     \hfill
      \begin{tabular}{@{}c@{}}
         \centering
         \includegraphics[width= 5.5 cm]{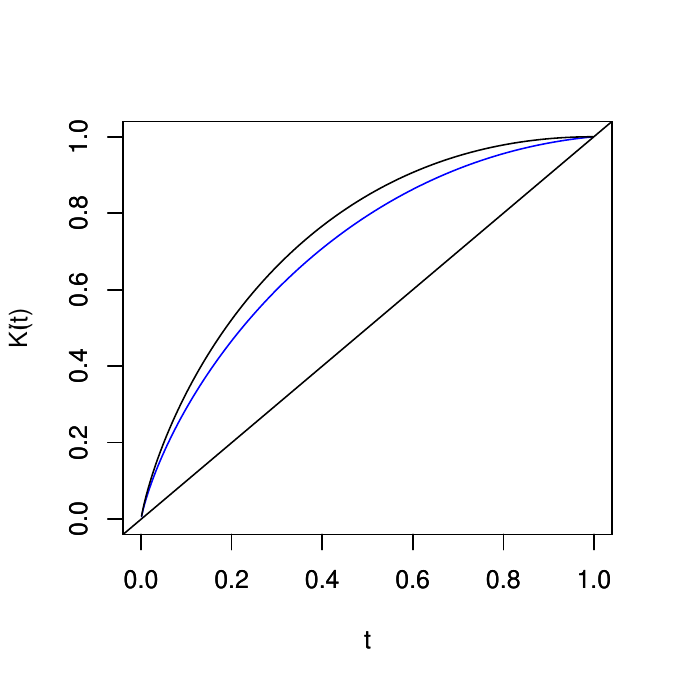}
     \end{tabular}
     \hfill
      \begin{tabular}{@{}c@{}}
         \centering
         \includegraphics[width= 5.5 cm]{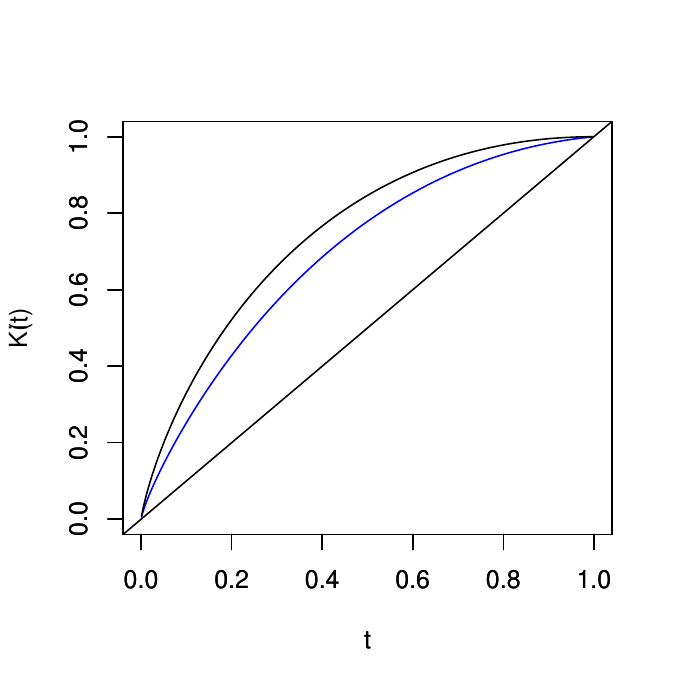}
     \end{tabular}
        \caption{
        Blue Curve: Kendall distribution function of (\ref{F13}).
        Black Curve: Independence Curve.
         Left: $\gamma = 0.01$; Center: $\gamma = 0.5$; Right: $\gamma = 2.5$.}
        \label{Figure15}
\end{figure}
\begin{figure}[h!]
     \centering
      \begin{tabular}{@{}c@{}}
         \centering
         \includegraphics[width= 5.5 cm]{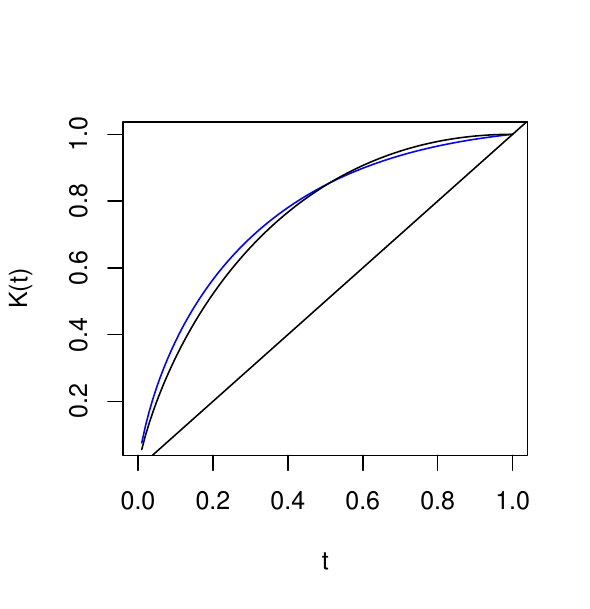}
     \end{tabular}
     \hfill
      \begin{tabular}{@{}c@{}}
         \centering
         \includegraphics[width= 5.5 cm]{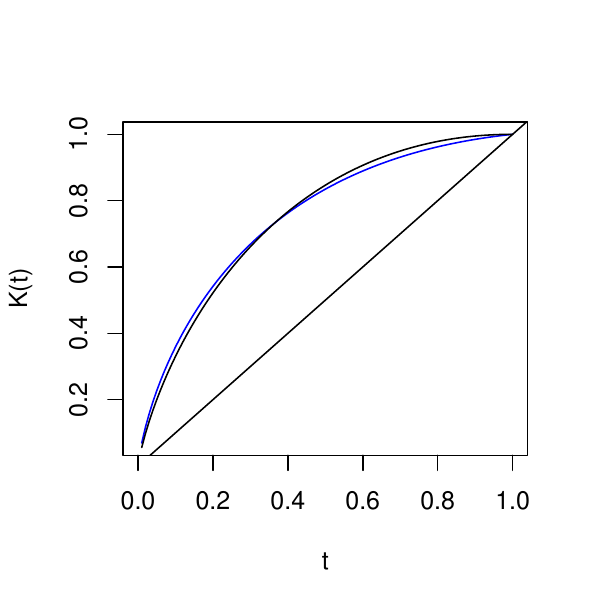}
     \end{tabular}
     \hfill
      \begin{tabular}{@{}c@{}}
         \centering
         \includegraphics[width= 5.5 cm]{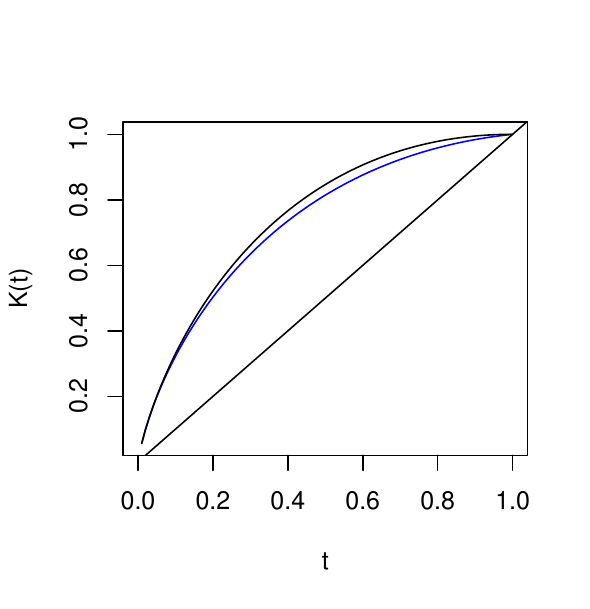}
     \end{tabular}
        \caption{
        Blue Curve: Kendall distribution Function of (\ref{F2}).
        Black Curve: independence curve.
         Left: $\gamma = 2$. Center: $\gamma = 3$. Right: $\gamma = 10$.}
        \label{Figure 52b}
\end{figure} 
\begin{figure}[h!]
     \centering
         \includegraphics[width= 4.5 cm]{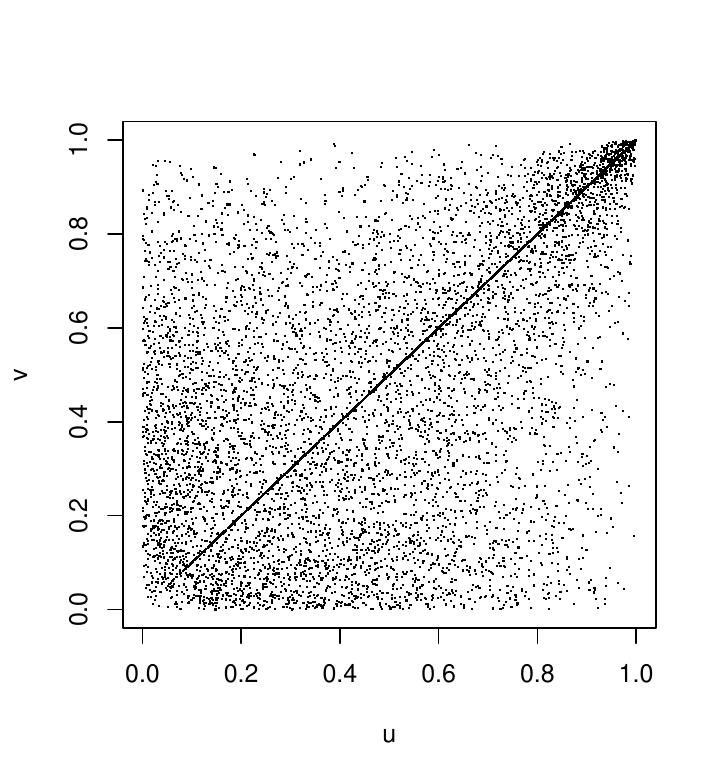} \hskip 2 cm
         \centering
         \includegraphics[width= 4.5 cm]{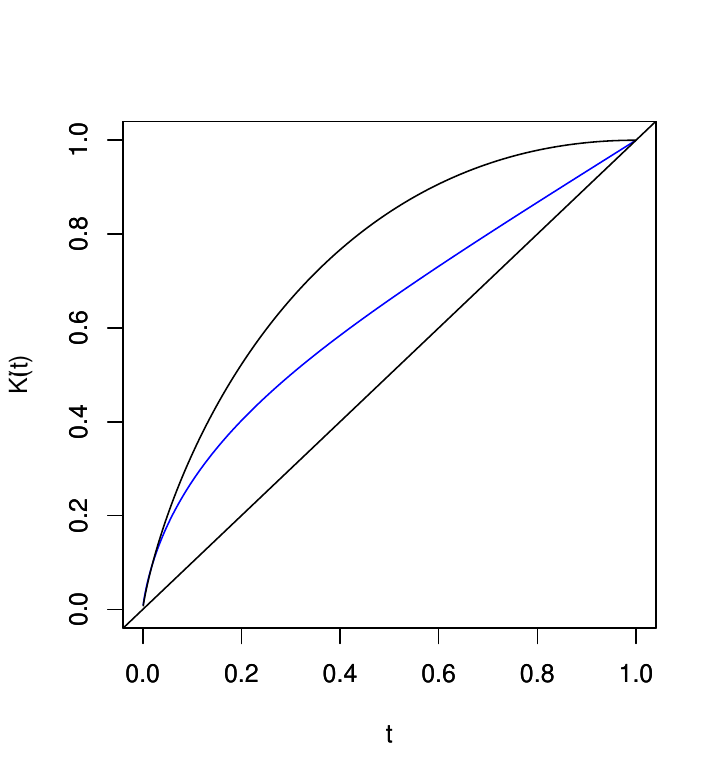}
        \caption{
        Left. Scatterplot from  (\ref{scatter survival}) with $\beta = 0.5$.
        Right. Blue Curve: Kendall distribution function of (\ref{scatter survival}) with $\beta = 0.5$. Black Curve: Independence Curve.
        The Kendall tau is equal to $0.483$, while the upper tail dependence coefficient is equal to $2-\sqrt{\frac{3}{2}}$}
        \label{Figure6}
\end{figure}
\begin{figure}[h!]
     \centering
         \includegraphics[width= 4.5 cm]{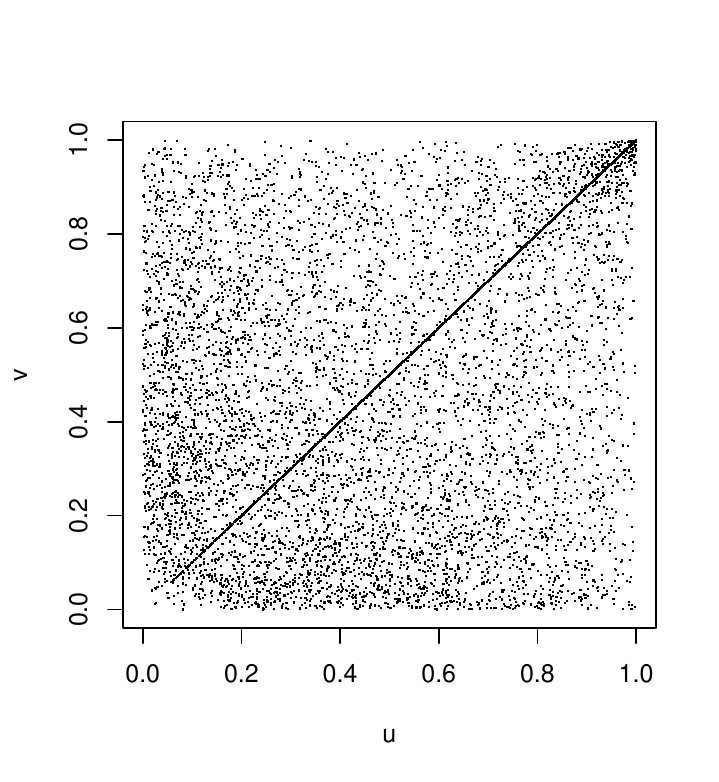} \hskip 2 cm
         \includegraphics[width= 4.5 cm]{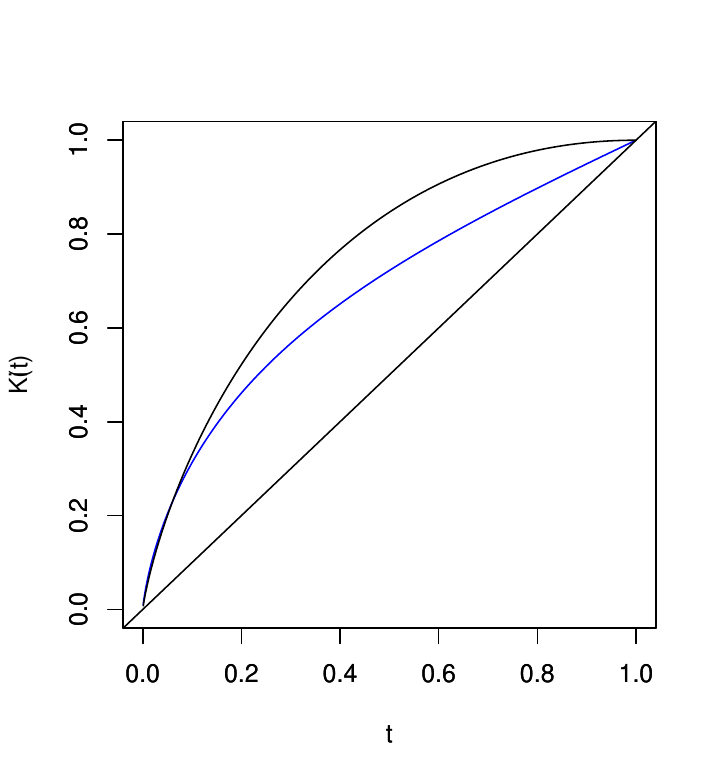} \hskip 2 cm
        \caption{
        Left. Scatterplot from  (\ref{scatter survival}) with $\beta = 0.75$.
        Right. Blue Curve: Kendall distribution function of (\ref{scatter survival}) with $\beta = 0.75$. Black Curve: Independence Curve.
        Kendall tau is equal to $0.305$, while the upper tail dependence coefficient is equal to $2-\frac{3}{2}^{\frac{3}{4}}$.}
        \label{Figure7}
\end{figure}
\begin{figure}[h!]
     \centering
         \includegraphics[width= 4.5 cm]{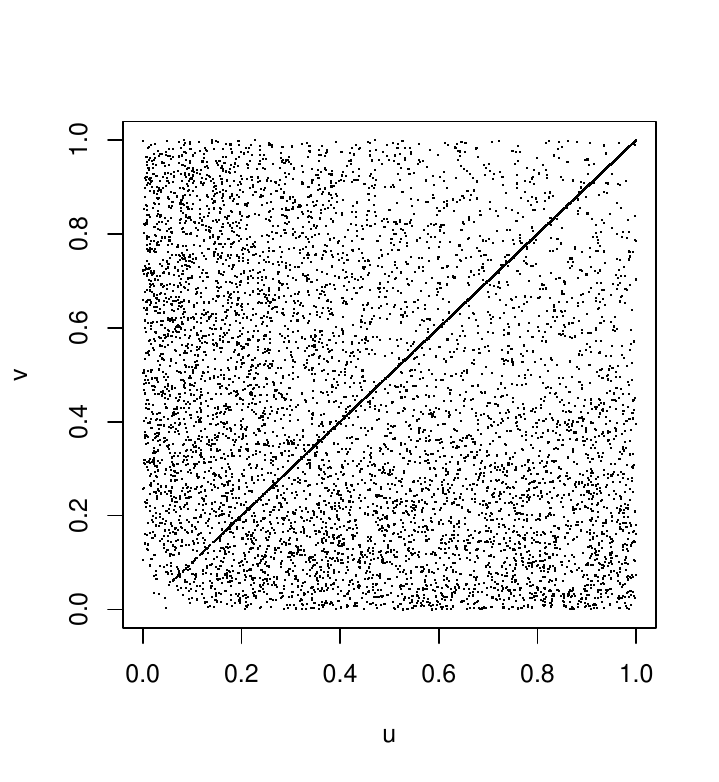} \hskip 2 cm
         \includegraphics[width= 4.5 cm]{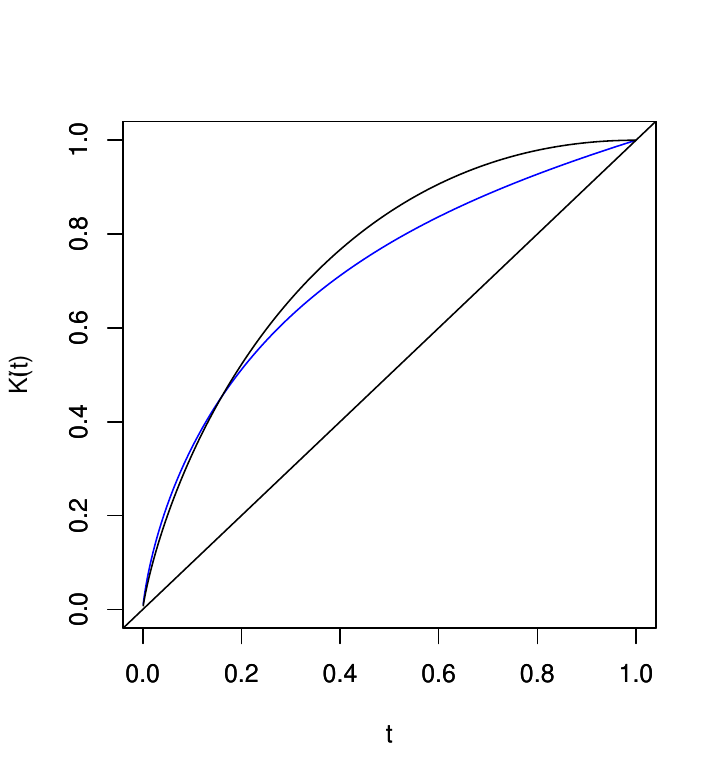}
        \caption{
        Left. Scatterplot from  (\ref{scatter survival}) with $\beta = 1$.
        Right. Blue Curve: Kendall distribution function of (\ref{scatter survival}) with $\beta = 1$. Black Curve: Independence Curve.
        Kendall tau is equal to $0.164$, while the upper tail dependence coefficient is equal to $\frac{1}{2}$.}
        \label{Figure8}
\end{figure}
\end{document}